\documentclass[10pt,twoside]{amsart}%
\usepackage{amsmath, amsthm, amscd, amsfonts, amssymb, graphicx, color, xcolor, soul }
\usepackage[bookmarksnumbered, colorlinks, plainpages]{hyperref}
\usepackage{graphicx}
\usepackage{epstopdf}
\sethlcolor{green}
\newcommand{\lcm}{\mathop{\mathrm{lcm}}}
\newtheorem{thm}{Theorem}[section]
\newtheorem{cor}[thm]{Corollary}
\newtheorem{lem}[thm]{Lemma}
\newtheorem{prop}[thm]{Proposition}
\newtheorem{defn}[thm]{Definition}

\newtheorem{ex}[thm]{Example}

\newtheorem*{thmA}{Theorem A}
\newtheorem*{corC}{Corollary C}
\newtheorem*{thmB}{Theorem B}
\newtheorem*{thmD}{Theorem D}
\newtheorem{proc}[thm]{}

\numberwithin{equation}{section}
\def\pn{\par\noindent}

\begin{document}

\vspace{1.3 cm}
\title{Quotient graphs for power graphs }
\author{D.~Bubboloni, Mohammad A.~Iranmanesh and S.~M.~Shaker }

\thanks{{\scriptsize
\hskip -0.4 true cm MSC(2010): Primary: 05C25; Secondary: 20B30.
}}
\maketitle


\begin{abstract} In a previous paper of the first author a procedure was developed for counting the components of a graph through the knowledge of the components of its quotient graphs. We apply here that procedure to the proper power graph  $\mathcal{P}_0(G)$ of a finite group $G$, finding a formula for the number $c(\mathcal{P}_0(G))$ of its components which is particularly illuminative when $G\leq S_n$  is a fusion controlled permutation group. We make use of the proper quotient power graph  $\widetilde{\mathcal{P}}_0(G)$, the proper order graph $\mathcal{O}_0(G)$ and the proper type graph $\mathcal{T}_0(G)$. We show that all those graphs are quotient of $\mathcal{P}_0(G)$ and
demonstrate a strong link between them dealing with $G=S_n$. We find simultaneously $c(\mathcal{P}_0(S_n))$ as well as the number of components of $\widetilde{\mathcal{P}}_0(S_n)$, $\mathcal{O}_0(S_n)$ and $\mathcal{T}_0(S_n)$.
\end{abstract}

\vskip 0.2 true cm
\pagestyle{myheadings}
\markboth{\rightline {\scriptsize  Bubboloni, Iranmanesh and Shaker}}
         {\leftline{\scriptsize Power graphs }}

\bigskip
\bigskip
\section{\bf Introduction and main results}
\vskip 0.4 true cm

 Kelarev and Quinn~\cite{kq} defined  the directed power graph $\overrightarrow{\mathcal{P}(S)}$ of a semigroup $S$ as the directed graph in which the set of vertices is $S$ and, for $x, y\in S$,  there is an arc  $(x,y)$ if $y=x^m$, for some $m\in\mathbb{N}$. The \emph{power graph} $\mathcal{P}(S)$ of a semigroup $S$, was defined by Chakrabarty, Ghosh and Sen~\cite{cgs} as the corresponding underlying undirected graph. They proved that for a finite group $G$, the power graph $\mathcal{P}(G)$ is complete if and only if $G$ is a cyclic group  of prime power order. In~\cite{c, cg} Cameron and Ghosh obtained interesting results about power graphs of finite groups, studying how the group of the graph automorphisms of $\mathcal{P}(G)$ affects the structure of the group $G$. Mirzargar, Ashrafi and Nadjafi~\cite{man} considered some further graph theoretical properties of the power graph $\mathcal{P}(G)$, such as the clique number, the independence number and the chromatic number and their relation to the group theoretical properties of $G$. Even though young, the theory of power graphs seems to be a very promising research area. The majority of its beautiful results dating before 2013 are collected in the survey \cite{sur}.

 In this paper we deal with the connectivity of $\mathcal{P}(G)$, where $G$ is a group.  All the groups considered in this paper are finite. Since it is obvious that $\mathcal{P}(G)$ is connected of diameter at most $2$, the focus is on  $2$-connectivity. Recall that a graph $X=(V_X,E_X)$  is $2$-connected if, for every $x\in V_X$, the $x$-deleted subgraph of $X$ is connected.
Thus $\mathcal{P}(G)$  is $2$-connected if and only if $\mathcal{P}_0 (G)$,  the $1$-deleted subgraph of $\mathcal{P} (G)$, is connected.  $\mathcal{P}_0 (G)$ is called the {\it proper power graph} of $G$ and our main aim is to find a formula for the number $c_0(G)$ of its components. We denote its vertex set $G\setminus\{1\}$ by $G_0.$ Recently Curtin,  Pourgholi and  Yousefi-Azari  \cite{pya2}  considered  the properties of the diameter of $\mathcal{P}_0 (G)$ and characterized the groups $G$ for which $\mathcal{P}_0 (G)$ is Eulerian. Moghaddamfar, Rahbariyan and Shi \cite{MRS}, found many relations between the group theoretical properties of $G$ and the graph theoretical properties of $\mathcal{P}_0 (G)$.
Here we apply the theory developed in \cite{bub} to get control on number and nature of the components of $P_0 (G)$ through those of some of its quotient graphs.

Throughout the paper, $n$ always indicates a natural number. The symmetric group $S_n$ and the alternating group $A_n$ are interpreted as naturally acting on the set $N=\{1,\dots,n\}$ and with identity element $id$.
 We are going to use notation and  definitions given in \cite{bub} about graphs and homomorphisms. In particular, every graph is finite, undirected, simple and reflexive, so that there is a loop on each vertex. The assumption about loops, which is not common for power graphs, is clearly very mild in treating connectivity and does not affect any result about components.

Up to now, the $2$-connectivity of $\mathcal{P}(G)$ has been studied for nilpotent groups and for some types of simple groups in ~\cite{pya}; for groups admitting a partition and for the symmetric and alternating groups, with a particular interest on the diameter of $\mathcal{P}_0 (G)$, in ~\cite{DF}. In those papers the results are obtained
through ingenious ad hoc arguments,  without developing a general method. The arguments often involve element orders  and,  when $G\leq S_n,$  the cycle decomposition of permutations.  We observed that what makes those ideas work is the existence of some quotient graphs for $\mathcal{P}_0 (G)$. For $\psi\in S_n$, let $T_{\psi}$ denote the type of $\psi$, that is, the partition of $n$ given by the lengths of  the orbits of $\psi$. Then
 there exists a quotient graph $\mathcal{O}_0(G)$ of $\mathcal{P}_0 (G)$ having vertex set $\{o(g): g\in G_0\}$ and, when $G$ is a permutation group, there exists a quotient graph $\mathcal{T}_0(G)$ of $\mathcal{P}_0 (G)$ having vertex set $\mathfrak{T}(G_0)$ given by the types $T_{\psi}$ of  the permutations $\psi\in G_0$ (Sections \ref{order-sec} and \ref{permutation}).

Recall that a homomorphism $f$ from the graph $X$ to the graph $Y$  is called complete if it maps both the vertices and edges of $X$ onto those of $Y$; tame if vertices with the same image are connected; locally surjective if it maps the neighborhood of each vertex  of $X$ onto the neighborhood in $Y$ of its image; orbit if the sets of vertices in $V_X$ sharing the same image coincide with the orbits of a group of automorphisms of $X$.
The starting point of our approach is to consider  the \emph{quotient power graph} $\widetilde{\mathcal{P}}_0(G)$, obtained from $\mathcal{P}_0 (G)$ by the identification of the vertices generating the same cyclic subgroup (Section \ref{onda}). The projection $\pi$ of $\mathcal{P}_0 (G)$ on $\widetilde{\mathcal{P}}_0(G)$ is tame  and thus the number  $\widetilde {c}_0(G)$ of  components of $\widetilde{\mathcal{P}}_0(G)$ is equal to $c_0(G)$. Moreover, both $\mathcal{O}_0(G)$ and $\mathcal{T}_0(G)$ may be seen also as quotients of $\widetilde{\mathcal{P}}_0(G)$, with a main difference between them.
 The projection $\widetilde{o}$ on $\mathcal{O}_0(G)$ is not, in general, locally surjective (Example \ref{no-fusion}) while,  for any $G\leq S_n$, the projection $\widetilde{t}$ on $\mathcal{T}_0(G)$ is complete and locally surjective (Propositions \ref{tau} and \ref{2con}).
As a consequence, while finding $c_0(G)$ through $\mathcal{O}_0(G)$ can be hard,  it is manageable  through $\mathcal{T}_0(G)$.
Call now
 $G\leq S_n$ fusion controlled if,  for every $\psi\in G$ and  $x\in S_n$ such that  $\psi^x\in G$, there exists $y\in N_{S_n}(G)$ such that   $\psi^x=\psi^y.$ Obviously $S_n$ and $A_n$ are both fusion controlled, but they are not the only examples. For instance,  if $n=mr$,  with $m,r\geq 2$, then the base group $G$ of the wreath product $S_m\wr S_r=N_{S_n}(G)$ is fusion controlled.
If $G$ is fusion controlled, then $\widetilde{t}$ is a complete orbit homomorphism (Proposition \ref{conj}) and hence  \cite[Theorem B] {bub} applies to $X=\widetilde{\mathcal{P}}_0(G)$ and  $Y=\mathcal{T}_0(G)$, giving an algorithmic method to get $c_0(G).$

In order to state our main results we need some further notation.
Denote by $\widetilde{\mathcal{C}}_0(G)$ the set of components of $\widetilde{\mathcal{P}}_0(G)$;  by $\mathcal{C}_0(\mathcal{T}(G))$ the set of components of $\mathcal{T}_0(G)$ and by $c_0(\mathcal{T}(G))$ their number;   by $c_0(\mathcal{O}(G))$ the number of components of $\mathcal{O}_0(G)$.
For $T\in  \mathcal{T}(G_0)$, denote by  $\mu_T(G)$ the number  of permutations of type $T$ in $G$; by $o(T)$ the order of any permutation having type $T;$ by $C(T)$ the component of $\mathcal{T}_0(G)$ containing $T$; by $\widetilde{\mathcal{C}}_0(G)_{T}$ the set of components of $\widetilde{\mathcal{P}}_0(G)$ in which there exists at least one vertex of type $T.$ Finally, for $C\in \widetilde{\mathcal{C}}_0(G)_{T}$, let $k_{C}(T)$ be the number of vertices in $C$ having type $T,$
and let $\phi$ denote the Euler totient function.

 \begin{thmA}  Let $G\leq S_n$ be fusion controlled. For $1\leq i\leq c_0(\mathcal{T}(G))$, let $T_i\in \mathfrak{T}(G_0)$ be such that $\mathcal{C}_0(\mathcal{T}(G))=\{C(T_i): i\in\{1,\dots,c_0(\mathcal{T}(G))\}\}$, and pick $C_i\in \widetilde{\mathcal{C}}_0(G)_{T_i}.$
Then
\begin{equation}\label{formula-fusion22}\displaystyle{c_0(G)=\sum_{i=1}^{c_0(\mathcal{T}(G))}\frac{\mu_{T_i} (G)}{\phi(o(T_i))k_{C_i}(T_i)}}.\end{equation}
 \end{thmA}

The connectivity properties of the graphs $\widetilde{\mathcal{P}}_0(G)$, $\mathcal{T}_0(G)$ and $\mathcal{O}_0(G)$ are strictly linked when $G\leq S_n$, especially when $G$ is fusion controlled.
 In the last section of the paper we consider $G=S_n$, finding $c_0(S_n)$, $c_0(\mathcal{T}(S_n))$ and $c_0(\mathcal{O}(S_n))$. In particular  we find, with a different approach, the values of $c_0(S_n)$  in \cite[Theorem 4.2]{DF}. Throughout the paper we denote by $P$ the set of prime numbers and put $P+1=\{x\in \mathbb{N}:  x=p+1 \ \hbox{for some}\   p\in P\}. $

 \begin{thmB} The values of $c_0(S_n)= \widetilde {c}_0(S_n), c_0(\mathcal{T}(S_n))$ and $c_0(\mathcal{O}(S_n))$  are as follows.

 \begin{itemize}

 \item[(i)] For $2\leq n\leq 7,$ they are given in Table \ref{eqtable1}
 below.
\\

\begin{table}[h]
\caption{$c_0(S_n), c_0(\mathcal{T}(S_n))$ and $c_0(\mathcal{O}(S_n))$  for $2\leq n \leq 7$.}\label{eqtable1}
\begin{center}

\begin{tabular}{|c|c|c|c|c|c|c|}
\hline
$n$ & $2 $&$ 3$ & $4$ & $5$ & $6$ &$ 7$\\
\hline
$c_0(S_n)$ &$ 1 $&$ 4$ & $13$ &$ 31$ & $83$ & $128$\\
\hline
$c_0(\mathcal{T}(S_n))$ & $1$ & $2$ &$ 3$ &$ 3$ &$ 4$&$ 3$\\
\hline
$c_0(\mathcal{O}(S_n))$ & $1$ &$ 2$ &$ 2$ & $2 $&$ 2$&$ 2$\\

\hline
\end{tabular}
\end{center}
\end{table}

\item[(ii)] For  $n\geq 8$, they are given by Table \ref{eqtable2} below, according to whether $n$ is a prime, one greater than a prime, or neither.
\end{itemize}

\begin{table}[ht]
\caption{$c_0(S_n)$, $c_0(\mathcal{T}(S_n))$ and $c_0(\mathcal{O}(S_n))$ for $ n \geq 8$}\label{eqtable2}
\begin{center}
\begin{tabular}{|c|c|c|c|c|}
\hline
$n$ & $n\in P$ & $n\in P+1$ & $n\notin P\cup (P+1)$ \\
\hline
$c_0(S_n)$ & $(n-2)!+1$ & $n(n-3)!+1$ & $1$ \\
\hline
$c_0(\mathcal{T}(S_n))=c_0(\mathcal{O}(S_n))$ & $2$ & $2$ & $1$  \\

\hline
\end{tabular}
\end{center}
\end{table}

 \end{thmB}

  \begin{corC} Let $n\geq 2.$ The following are equivalent:
 \begin{itemize}
 \item[(i)] $\mathcal{P}(S_n)$ is $2$-connected;
\item[(ii)] $\mathcal{P}_0(S_n)$ is connected;
 \item[(iii)] $\widetilde{\mathcal{P}}_0(S_n)$ is connected;
\item[(iv)] $\mathcal{T}_0(S_n)$ is connected;
 \item[(v)] $\mathcal{O}_0(S_n)$ is connected;
 \item[(vi)] $n=2$ or $n\in \mathbb{N}\setminus [P\cup(P+1)].$
\end{itemize}
 \end{corC}

Observe that $\mathcal{T}_0(S_n)$  has a purely arithmetic  interest, because $\mathfrak{T}(S_n)$ is the whole set of the partition of $n.$
Going beyond a mere counting, we describe the components of the  graphs belonging to  $\mathcal {G}_0=\{\mathcal{P}_0(S_n),\widetilde{\mathcal{P}}_0(S_n), \mathcal{T}_0(S_n), \mathcal{O}_0(S_n)\}. $
To start with, note that in the connected case $n\in \mathbb{N}\setminus [P\cup(P+1)]$, no $X_0\in \mathcal {G}_0$ is  a complete graph because $X_0$ admits as quotient $\mathcal{O}_0(S_n)$ which is surely incomplete having as vertices at least two primes. For $\psi\in S_n\setminus\{id\}$ vertex of $\mathcal{P}_0 (S_n)$, let $[\psi]$ denote the corresponding vertex of  $\widetilde{\mathcal{P}}_0(S_n)$.

 \begin{thmD}  Let $n\geq 8$, with $n\in\{p,p+1\}$ for some prime $p$. Let $k=(n-2)!$ if $n=p$ and let $k=n(n-3)!$ if $n=p+1.$ Let $\Delta_n$ be the set  of every nonidentity permutation in $S_n$ which is not a $p$-cycle.

     \begin{itemize}
      \item[(i)] $\mathcal{P}_0(S_n)$ consists of the main component induced on $\Delta_n$ and $k$-many  complete components, each comprised of $(p-1)$-many $p$-cycles.
   \item[(ii)] $\widetilde{\mathcal{P}}_0(S_n)$ consists of the main component 
    induced on $\{[\psi]:\psi\in \Delta_n\}$ and $k$-many
   isolated vertices.
   \item[(iii)] $\mathcal{T}_0(S_n)$ consists of the main component 
    induced on $\{T_{\psi}:\psi\in \Delta_n\}$ and the component containing the type of a $p$-cycle which is  an isolated vertex.
 \item[(iv)] $\mathcal{O}_0(S_n)$ consists of the main component 
  induced by $\{o(\psi): \psi\in \Delta_n\}$ and the component given by the isolated vertex $p.$
 \end{itemize}
 In all the above cases, the main component is never complete.
\end{thmD}

Complete  information about the components of the graphs in $\mathcal{G}_0$, for $3\leq n\leq 7$, can be found within the proof of Theorem B\,(i), taking into account Lemma \ref{component-quotient} for  $\mathcal{P}_0(S_n)$.
In particular, looking at the details, one easily checks that all the components  of $X_0 \in\mathcal{G}_0$ apart from one are isolated vertices (complete graphs when $X_0=\mathcal{P}_0(S_n)$) if and only if $n\geq 8$ or $n=2.$

In a forthcoming paper  \cite{BIS2} we will treat the alternating group $A_n$ computing $c_0(A_n)$, $c_0(\mathcal{T}(A_n))$ and $c_0(\mathcal{O}(A_n)).$ We will also correct some mistakes about $c_0(A_n)$ found in \cite{DF}. We believe that our algorithmic method \cite[Theorem B]{bub} may help, more generally, to obtain $c_0(G)$ where $G$ is simple and almost simple. This, in particular, could give an answer to the interesting problem of classifying all the simple groups with $2$-connected power graph, posed in \cite[Question 2.10]{pya}. About that problem, in  \cite{BIS2}  we show that there exist infinite examples of alternating groups with $2$-connected power graph and that $A_{16}$  is that of smaller degree.

\vskip 0.6 true cm
\section{\bf Graphs}\label{graphs}
\vskip 0.4 true cm
For a finite set $A$ and $k\in \mathbb{N}$, let $\binom{A}{k}$ be the set of the subsets of $A$ of size $k.$ In this paper,  as in \cite{bub}, a graph
$X=(V_X,E_X)$ is a pair of finite sets such that $V_X\neq\varnothing$ is the set of vertices, and $E_X$
 is the set of edges which is the union of the set of  loops $L_X=\binom{V_X}{1}$
and a set of  proper edges $E^*_X\subseteq \binom{V_X}{2}$. We usually specify the edges of a graph $X$ giving only $E^*_X$.

 Paper \cite{bub}  is the main reference for the present paper. For the general information about graphs see  \cite[Section 2]{bub}. Recall that, for a graph $X$, $\mathcal{C}(X)$ denotes the set of components of $X$ and $c(X)$ their number. If $x\in V_X$, the component of $X$ containing $x$ is denoted by $C_X(x)$ or more simply, when there is no risk of misunderstanding, by $C(x)$.
For $s\in\mathbb{N}\cup\{0\}$, a subgraph $\gamma$ of $X$ such that $V_{\gamma}=\{x_i: 0\leq i\leq s\}$ with distinct $x_i\in V_X$ and $E^*_{\gamma}=\{\{x_i,x_{i+1}\} : 0\leq i\leq s-1\}$, is called a  path of length $s$ between $x_0$ and $x_s$, and will be simply denoted  by the ordered list $x_0,\dots, x_s$ of its vertices.

 For the formal definitions of surjective, complete, tame, locally surjective, pseudo-covering, orbit and component equitable homomorphism and the notation for the corresponding sets of homomorphisms see  \cite[Section 4.2, Definitions  5.3, 5.7, 4.4, 6.4]{bub}.

By \cite[Propositions 5.9 and 6.9]{bub}, we have that
 \begin{equation}\label{general-inc}
 \mathrm{O}(X,Y)\cap  \mathrm{Com}(X,Y)\subseteq \mathrm{LSur}(X,Y)\cap \mathrm{CE}(X,Y)\cap \mathrm{Com}(X,Y)
 \end{equation}
The content  of \cite{bub} is all we need to conduce our arguments up to just a couple of definitions and related results.
 \begin{defn}\label{q-2hom} {\rm Let $X$ and $Y$ be graphs, and  $f\in\mathrm{Hom}(X,Y).$  Then $f$ is called a $2$-{\it homomorphism} if, for every $e\in E^*_X$, $f(e)\in E^*_Y$. We denote the set of the $2$-homomorphisms from $X$ to $Y$ by $2 \mathrm{Hom}(X,Y)$.}
 \end{defn}
From that definition we immediately deduce the following lemma.
\begin{lem}\label{2hom-iso} Let $f\in 2\mathrm{Hom}(X,Y)$ and $x\in V_X$. If $f(x)$ is isolated in $Y,$ then $x$ is isolated in $X$.
 \end{lem}
 \begin{defn}\label{deleted} {\rm
 Let $X$ be a graph.
 If $x_0\in V_X$, then the {\it $x_0$-deleted subgraph} $X-x_0$ is defined as the subgraph of $X$ with vertex set $V_X\setminus\{x_0\}$ and edge set given by the edges in $E_X$ not incident to $x_0$.
$X$ is called {\it $2$-connected} if, for every $x_0\in V_X$, $X-x_0$ is connected.}
 \end{defn}
 To deal with vertex deleted subgraphs and quotient graphs, we will use the following lemma several times.
\begin{lem}\label{cutgraph-hom} Let $f\in\mathrm{Hom}(X,Y)$.
\begin{itemize}
\item[(i)] Suppose $x_0\in V_X$ is such that $f^{-1}(f(x_0))=\{x_0\}$. Then $f$ induces naturally  $f_{x_0}\in\mathrm{Hom}( X-x_0,Y-f(x_0)).$
Moreover, if $f$ is surjective (complete, pseudo-covering), then also $f_{x_0}$ is surjective (complete, pseudo-covering).
\item[(ii)] Let $\sim$ be an equivalence relation on $V_X$ such that, for each $x_1,x_2\in V_X$, $x_1\sim x_2$ implies $f(x_1)=f(x_2).$ Then the map $\widetilde{f}:[V_X]\rightarrow V_Y$, defined by $\widetilde{f}([x])=f(x)$ for all $[x]\in[V_X]$, is a homomorphism from $X/\hspace{-1mm}\sim$ to $Y$ such that $\widetilde{f}\circ \pi=f.$
If $f$ is surjective (complete, pseudo-covering), then also $\widetilde{f}$ is  surjective (complete, pseudo-covering).
\end{itemize}

\end{lem}
\begin{proof} (i)  Since $f(V_X\setminus\{x_0\})\subseteq V_Y\setminus\{f(x_0)\}$,  we  can consider the map $f_{x_0}: V_X\setminus\{x_0\}\rightarrow V_Y\setminus\{f(x_0)\}$, defined by $f_{x_0}(x)=f(x)$ for all $x\in V_X\setminus\{x_0\}.$
We show that $f_{x_0}$ defines a homomorphism. Pick $e\in E_{X-x_0}$, so that $e=\{x_1,x_2\}\in E_X$  for suitable $x_1,x_2\in V_X$ with  $x_1,x_2\neq x_0$.
By  $f^{-1}(f(x_0))=\{x_0\}$, we get $f(x_1), f(x_2)\neq f(x_0)$ and thus, since $f$ is a homomorphism, we get $\{f(x_1), f(x_2)\}\in E_{Y-f(x_0)}.$
If $f$ is surjective, clearly $f_{x_0}$ is surjective.
Assume now that $f$ is complete  and show that $f_{x_0}$ is complete. Let $\{f_{x_0}(x_1), f_{x_0}(x_2)\}\in E_{Y-f(x_0)}.$ Then $\{f(x_1), f(x_2)\}\in E_Y,$ with $f(x_1), f(x_2)\neq f(x_0)$.
Since $f$ is complete, there exists $\overline{x}_1, \overline{x}_2\in V$ such that $f(\overline{x}_1)=f(x_1),\  f(\overline{x}_2)=f(x_2)$ and $\{\overline{x}_1, \overline{x}_2\}\in E_X.$ From $f(x_1), f(x_2)\neq f(x_0)$ we deduce that $\overline{x}_1, \overline{x}_2\neq x_0.$ Thus $\overline{x}_1, \overline{x}_2\in V_{X-x_0}$ and $\{\overline{x}_1, \overline{x}_2\}\in E_{X-x_0}.$  An obvious adaptation of this argument works also in the pseudo-covering case.

(ii) The fact that $\widetilde{f}$ is a homomorphism such that $\widetilde{f}\circ \pi=f$ is the content of \cite[Theorem 1.6.10]{kna}. Assume that $f$ is surjective. Then, by $\widetilde{f}\circ \pi=f$,  $\widetilde{f}$ is surjective too.
Assume now that $f$ is complete  and show that $\widetilde{f}$ is complete. By what observed above, $\widetilde{f}$ is surjective.
 Let $e=\{\widetilde{f}([x_1]), \widetilde{f}([x_2])\}=\{f(x_1), f(x_2)\}\in E_Y$.
Since $f$ is complete, there exists $\overline{x}_1, \overline{x}_2\in V_X$ such that $f(\overline{x}_1)=f(x_1),  f(\overline{x}_2)=f(x_2)$ and $\{\overline{x}_1, \overline{x}_2\}\in E_X.$  Then also $e'=\{[\overline{x}_1], [\overline{x}_2]\}\in [E_X]$ and $e=\widetilde{f}(e').$ An obvious adaptation of this argument works also in the pseudo-covering case.
\end{proof}

When no ambiguity arises, we denote the map $f_{x_0}$ of the above lemma again by $f.$

In the following, whatever the deleted vertex $x_0\in V_X$ is, we write $X_0=((V_X)_0,(E_X)_0)$ for the $x_0$-deleted subgraph. Moreover  we write $\mathcal{C}_0(X)$ for the set of the components of $X_0$ as well as $c_0(X)$  for their number. This  helps to standardize the notation throughout the paper. Some abbreviations will sometimes be introduced. The terminology not  explicitly introduced is standard and can be found in \cite{dl}.

\vskip 1 true cm
\section{\bf Power graphs}\label{onda}
\vskip 0.4 true cm
Throughout the next sections, let $G$ be a finite group with identity element $1$ and let $G_0=G\setminus\{1\}$. For $x\in G,$ denote by $o(x)$ the order of $x.$
\vskip 0.4 true cm
\begin{defn}\label{pow-gr}{\rm   The
{\it power graph} of $G$ is the graph $\mathcal{P} (G)=(G,E)$ where, for $x,y\in G$, $\{x,y\}\in E$ if there exists $m\in\mathbb{N}$ such that $x=y^m$ or $y=x^m$.
The {\it proper  power graph} $\mathcal{P}_0 (G)=(G_0,E_0)$ is defined as the $1$-deleted subgraph  of $\mathcal{P}(G).$}
\end{defn}

To deal with the graphs $\mathcal{P}(G)$ and $\mathcal{P}_0 (G)$ and simplify their complexity, we start considering the corresponding quotient graphs in which the elements of $G,$ generating the same cyclic subgroup, are identified  in a unique vertex.

\begin{defn}\label{qpow-gr}{\rm Define  for $x,y\in G$, $x\sim y$ if $\langle x\rangle=\langle y\rangle$. Then $\sim$ is an equivalence relation on $G$ and the equivalence class of $x\in G$, $[x]=\{x^m : 1\leq m\leq o(x),\,\gcd (m,o(x))=1\}$ has size $\phi(o(x))$.
 The quotient graph $\mathcal{P}(G)/\hspace{-1mm}\sim=([G],[E])$  is called the\emph{ quotient power graph} of $G$ and denoted  by $\widetilde{\mathcal{P}}(G)$.}
\end{defn}

By definition of quotient graph,  the vertex set of   $\mathcal{P}(G)/\hspace{-1mm}\sim$ is $[G]=\{[x] : x\in G\}$ and $\{[x],[y]\}\in [E]$ is an edge in $\widetilde{\mathcal{P}}(G)$  if there exist $\widetilde{x}\in [x]$ and $\widetilde{y}\in[y]$ such that $\{\widetilde{x},\widetilde{y}\}\in E$, that is, $\widetilde{x},\widetilde{y}$ are one the power of the other.

\begin{lem}\label{lato} For every $x,y\in G$, $\{[x], [y]\}\in [E]$
 if and only if $\{x,y\}\in E$.
\end{lem}
\begin{proof}  Let $x,y\in G$ such that $\{[x], [y]\}$ is an edge in $\widetilde{\mathcal{P}}(G)$. Then there exist $\tilde{x}\in [x]$ and $\tilde{y}\in[y]$ such that one of them is a power of the other. To fix the ideas, let $\tilde{x}=(\tilde{y})^m$, for some $m\in\mathbb{N}.$ Since $x\in \langle \tilde{x}\rangle$ and $\tilde{y}\in \langle y\rangle$, there exist $a,b\in\mathbb{N}$ such that $x=(\tilde{x})^a$ and $\tilde{y}=y^b$. It follows that $x=y^{abm}$ and thus $\{x,y\}\in E$. The converse is trivial.
\end{proof}
The above lemma  may be thought of as saying  that the projection of $\mathcal{P}(G)$ onto its quotient $\widetilde{\mathcal{P}}(G)$ is a strong homomorphism in the sense of \cite[Definition 1.5]{kna-paper}.

\begin{defn}\label{qpropow-gr}{\rm The $[1]$-deleted subgraph of $\widetilde{\mathcal{P}}(G)$ is called
the  \emph{proper quotient power graph}  of $G$ and denoted by $\widetilde{\mathcal{P}}_0(G)=([G]_0,[E]_0)$. }
\end{defn}

Since $[x]=[1]$ if and only if $x=1,$ applying Lemma \ref{cutgraph-hom} and \cite[Lemma 4.3]{bub}  to the projection of $\mathcal{P}(G)$ onto $\widetilde{\mathcal{P}}(G)$ gives that $\widetilde{\mathcal{P}}_0(G)$ is equal to the quotient graph $\mathcal{P}_0(G)/\hspace{-1mm}\sim$.  For short, we  denote the set of components of $\mathcal{P}_0(G)$ by $\mathcal{C}_0 (G)$ and their number by $c_0(G)$. Similarly we denote the set of components of $\widetilde{\mathcal{P}}_0(G)$ by $\widetilde{\mathcal{C}}_0 (G)$ and their number by $\widetilde{c}_0(G)=c(\widetilde{\mathcal{P}}_0(G)).$ Lemma \ref{lato} immediately extends to $\widetilde{\mathcal{P}}_0(G)$.

 \begin{lem}\label{lato2} For every $x,y\in G_0,$ $\{[x], [y]\}\in [E]_0$ if and only if $\{x,y\}\in E_0$.
\end{lem}

\begin{lem}\label{rmk:1}  The graph  $\widetilde{\mathcal{P}}_0(G)$  is a tame and pseudo-covered
 quotient of $P_0 (G)$. In particular $c_0(G)=\widetilde{c}_0(G).$
\end{lem}
\begin{proof}  Let $x,y\in G_0$ such that $x\sim y$. Then $y=x^m$ for some $m\in \mathbb{N}$ and thus $\{x,y\}\in E.$ This shows that the quotient graph $\widetilde{\mathcal{P}}_0(G)$ is tame (\cite[Definition 3.1]{bub}). Then, \cite[Proposition 3.2]{bub} applies giving $c_0(G)=\widetilde{c}_0(G).$ The fact that the quotient is pseudo-covered (\cite[Definition 5.14]{bub}) is an immediate consequence of Lemma \ref{lato2}.
\end{proof}
\begin{lem}\label{component-quotient}  Let $\pi$ be the projection of $ \mathcal{P}_0(G)$ on $\widetilde{\mathcal{P}}_0(G) $.  Then,  the map from $\mathcal{C}_0 (G)$ to $\widetilde{\mathcal{C}}_0 (G) $ which associates, with every $C\in \mathcal{C}_0 (G),$ the component $\pi(C) $ is a bijection. Given $\widetilde{C}\in \widetilde{\mathcal{C}}_0 (G) $, the set  of vertices  of the unique $C\in \mathcal{C}_0 (G)$ such that $\pi(C)=\widetilde{C}$ is given by $\pi^{-1}(V_{\widetilde{C}})$.
\end{lem}
\begin{proof} By Lemma \ref{rmk:1}, $\pi$ is tame and pseudo-covering. Thus we apply \cite[Corollary 5.13]{bub} to $\pi$.
\end{proof}
In the recent paper \cite{IS} the second and third author have investigated the groups $G$ for which $\widetilde{\mathcal{P}}(G)$ is a tree and for which $\widetilde{\mathcal{P}}_0(G)$ is a path or a bipartite graph.
\vskip 0.6 true cm
\section{\bf Order graphs}\label{order-sec}
\vskip 0.4 true cm
Let $O(G)=\{ o(g): g\in G \}$. The map $o:G\rightarrow O(G)$, associating to every $x\in G$ its order $o(x)$, is called the {\it order map} on $G$. We say that $m\in \mathbb{N}$ is a proper divisor of $n\in \mathbb{N}$ if $m\mid n$ and $m\notin\{1,n\}$.

\begin{defn}\label{O}{\rm The {\it order graph} of $G$ is the graph $\mathcal{O}(G)$ with vertex set $O(G)$ and edge set $E_{\mathcal{ O}(G)}$ where, for every $m,n\in O(G)$, $\{m,n\}\in E_{\mathcal{ O}(G)}$ if $m\mid n$ or $n\mid m$.
The {\it proper order graph}  $\mathcal {O}_0(G)$  is defined as the $1$-deleted graph of $\mathcal{O}(G).$ Its vertex set is then $O_0(G)=O(G)\setminus \{1\}.$}
\end{defn}

Note that $\{m,n\}\in E^*_{\mathcal{ O}_0(G)}$ only if one of $m$ and $n$ is a proper divisor of the other.
Clearly $\mathcal{O}(G)$  is always connected and it is $2$-connected if and only if $\mathcal {O}_0(G)$  is connected.
\begin{prop}\label{order}  Let $G$ be a group. The order map defines a complete homomorphism
$o :\mathcal{P}(G)\rightarrow \mathcal{O}(G)$ which induces a complete homomorphism $o:\mathcal{P}_0(G)\rightarrow \mathcal{O}_0(G)$, and a complete $2$-homomorphism
$\widetilde{o}:\widetilde{\mathcal{P}}_0(G)\rightarrow \mathcal{O}_0(G).$
If $G$ is cyclic, $\widetilde{o}$ is an isomorphism.
\end{prop}

\begin{proof}  For every $m\in \mathbb{N}$,  $o(x^m)$ is a divisor of $o(x),$ so  $o$ is a surjective homomorphism. We show that $o$ is complete.  Let $e=\{o(x),o(y)\}\in E_{\mathcal{ O}(G)},$ for some $x,y\in G$. Then, without loss of generality, we may assume that $o(y)\mid o(x).$ Since in $\langle x\rangle$ there exist elements of each order dividing $o(x),$ there exists $\overline{y}\in\langle x\rangle$ with $o(\overline{y})=o(y)$. Let $m\in\mathbb{N}$ be such that  $\overline{y}=x^m.$ Then $\{x,\overline{y}\}\in E$ and $o(\{x,\overline{y}\})=e.$ Now, since $o(x)=1$ if and only if $x=1$, Lemma \ref{cutgraph-hom} applies giving the desired result both for the vertex deleted graph  and  for the quotient graph. We are left to check that $\widetilde{o}$ is a $2$-homomorphism. Let $\{[x],[y]\}\in [E]^*_0$ and show that $\widetilde{o}([x])\neq\widetilde{o}([y])$. Assume the contrary, that is, $o(x)=o(y).$ By Lemma \ref{lato2}, we have $\{x,y\}\in E_0$ so that $x$ and $y$  are one the power of the other. It follows that $\langle x\rangle=\langle y\rangle$, against $[x]\neq [y].$

Finally let $G$ be cyclic. To prove that $\widetilde{o}$ is an isomorphism, it is enough to show that $\widetilde{o}$ is injective. Assume that for some $[x],[y] \in [G]_0$ we have $\widetilde{o}([x])=\widetilde{o}([y])$, that is, $o(x)=o(y)=m.$ Since in a cyclic group there exists exactly one subgroup for each $m\mid |G|,$ we deduce that $\langle x\rangle= \langle y\rangle$ and so $[x]=[y].$
\end{proof}
An application of \cite[Lemma 4.3]{bub} gives, in particular, the following result.
\begin{cor}\label{quotient-order} For each finite group $G$, the graph $\mathcal {O}_0(G)$ is a quotient of the graph $\widetilde{\mathcal{P}}_0(G).$
\end{cor}
 We exhibit now an example showing that, in general, $\widetilde{o}$ is not pseudo-covering.
  \begin{ex} \label{no-fusion}{\rm Let $G$ be the $2$-Sylow subgroup of $S_4$   given by
 \[G=\{id, (1\ 3), (2\ 4), (1\ 3)(2\ 4), (1\ 2)(3\ 4), (1\ 4)(2\ 3), (1\ 2\ 3\ 4), (1\ 4\ 3\ 2)\}.\]
Then $\mathcal{O}_0(G)$ is reduced to a path of length $1$ between the only two vertices $2$ and $4,$ while $\widetilde{\mathcal{P}}_0(G)$ has $6$ vertices and $5$ components because  the vertices $[(1\ 3)],$ $ [(2\ 4)]$, $ [(1\ 2)(3\ 4)],$ $[(1\ 4)(2\ 3)]$ are isolated while $\{[(1\ 2\ 3\ 4)], [(1\ 3)(2\ 4)]\}$ is an edge. $\widetilde{o}$ takes the component having as only vertex $[(1\ 3)]$ onto the subgraph $(2,\{2\})$, which is not a component of $\mathcal{O}_0(G)$. By \cite[Theorem A\,(i)]{bub}, this guarantees that $\widetilde{o}$ is not pseudo-covering. }
 \end{ex}

The above example indicates that the reduction of complexity obtained passing from the proper power graph to the proper order graph, is usually too strong. For instance, if $G$ is the group  in the previous example and $C_4$ the cyclic group of order $4$, we have $\mathcal{O}_0(G)\cong\mathcal{O}_0(C_4).$ In particular we cannot hope, in general, to count the components of  $\widetilde{\mathcal{P}}_0(G)$ relying on those of $\mathcal{O}_0(G).$ Anyway, taking into account the graph $\mathcal{O}_0(G)$, we get useful information on the isolated vertices of  $\widetilde{\mathcal{P}}_0(G)$.
\begin{lem}\label{isolated-order} Let $x\in G_0$.
\begin{itemize}
\item[(i)] If $o(x)\in O_0(G)$ is isolated in $\mathcal{O}_0(G)$, then $[x]$ is isolated in $\widetilde{\mathcal{P}}_0(G).$
\item[(ii)] If $[x]$ is isolated in $\widetilde{\mathcal{P}}_0(G),$ then $o(x)$ is a prime and the component of $\mathcal{P}_0(G)$ containing $x$ is a complete graph on $p-1$ vertices.
\item[(iii)] $c_0(\mathcal{O}(G))\leq |\{p\in P: p\mid |G|\}|.$
\end{itemize}
\end{lem}
\begin{proof} (i) Apply Lemma \ref{2hom-iso} to the $2$-homomorphism $\widetilde{o}$ of Proposition \ref{order}.

(ii) Let $[x]$ be isolated in $\widetilde{\mathcal{P}}_0(G)$. We first show that $o(x)=p$, for some prime $p$.
 Assume, by contradiction, that $o(x)$ is composite. Then there exists $k\in\mathbb{N}$ such that $1\neq \langle x^k\rangle\neq \langle x\rangle$ and so $\{[x],[x^k]\}\in [E]^*_0$, contradicting $[x]$ isolated. Let next $C=C_{\mathcal{P}_0(G)}(x)$. By \cite[Theorem A\,(i)]{bub} applied to the pseudo-covering projection $\pi: \mathcal{P}_0(G)\rightarrow \widetilde{\mathcal{P}}_0(G)$, we have that $\pi(C)=C_{\widetilde{\mathcal{P}}_0(G)}([x])$ and thus $\pi(C)$ is reduced to the vertex $[x]$.
So, if $x'\in V_C$ we have that $[x']=[x]$ that is $\langle x'\rangle=\langle x \rangle$. It follows that $V_C$ is the set of generators of the cyclic group $\langle x\rangle$ of order $p$. Thus $|V_C|=p-1$ and $C$ is a complete graph on $p-1$ vertices.

(iii) Let $m\in O_0(G)$. If $p$ is a prime dividing $m$, then $p\in O_0(G)$ and $p\mid |G|.$ Moreover,
 $\{m,p\}\in E_{\mathcal{O}_0(G)}$ so that $m\in C_{\mathcal{O}_0(G)}(p).$
\end{proof}

\vskip 0.6 true cm
\section{\bf {\bf \em{\bf Quotient graphs of power graphs associated with permutation groups}}}\label{permutation}
\vskip 0.4 true cm
Let $G\leq S_n$ be a permutation group  of degree $n\in\mathbb{N},n\geq 2$ acting naturally on $N=\{1,\dots,n\}$.
We want  to  determine $c_0(G)$ from $X=\widetilde{\mathcal{P}}_0(G)$ by suitable quotients. We need to find a graph $Y$ and a homomorphism $f\in \mathrm{PC}(X,Y)=\mathrm{LSur}(X,Y)\cap \mathrm{Com}(X,Y)$ to which to apply  Formula $(1.1)$ in \cite[Theorem A]{bub}  or, better,  a homomorphism $f\in \mathrm{O}(X,Y)\cap \mathrm{Com}(X,Y)$ to which we can apply the Procedure \cite[6.10]{bub}.
 To start with, we need to associate to every permutation an arithmetic object.
\subsection{\bf {\bf \em{\bf Partitions }}}\label{partition}
\vskip 0.4 true cm
Let $n,r\in \mathbb{N},$  with $n\geq r.$ An
$r$-{\em partition} of $n$ is an unordered $r$-tuple
$T=[x_1,\dots,x_r]$, with $x_i\in\mathbb{N}$ for every $i\in
\{1,\dots,r\},$ such that $n=\sum_{i=1}^{r}x_i.$
The $x_i$ are called the {\em terms} of the partition; the  {\it order} of $T$  is defined by $o(T)=\lcm\{x_i\}_{i=1}^r.$
We denote by $\mathfrak{T}_r(n)$  the set of  the
$r$-partitions of $n$  and we call each element in  $\mathfrak{T}(n)=\bigcup_{r=1}^n \mathfrak{T}_r(n)$ a {\em partition} of $n.$
Given $T\in \mathfrak{T}(n)$, let $m_1<\dots < m_k$ be its $k\in \mathbb{N}$ distinct terms; if $m_j$ appears $t_j\geq 1$ times in $T$ we use the notation $T=[m_1^{t_1},..., m_k^{t_k}]$. Moreover, we say that $[m_1^{t_1},..., m_k^{t_k}]$ is
the {\it normal form} of $T$ and that $t_j$ is the {\it multiplicity} of the term $m_j$.  We will accept, in some occasions,  the multiplicity $t_j=0$ simply  to say that a certain natural number $m_j$ does not appear as a  term in $T.$  We usually do not write the multiplicities equal to $ 1$. For instance the partition $[1,1,3]$ can be written $[1^2,3^1]$  or $[1^2,3]$ or $[1^2,2^0,3]$. The partition $[1^n]$ is called the {\it trivial partition}. We put $\mathfrak{T}_0(n)=\mathfrak{T}(n)\setminus \{[1^n]\}$.
\subsection{\bf {\bf \em{\bf Types of permutations }}}\label{types}
\vskip 0.4 true cm
Let $\psi\in S_n.$ The {\sl type}  of  $\psi$ is the partition of $n$ given by  the unordered list
$T_{\psi}=[x_1,...,x_r]$ of the sizes $x_i$  of the $r\in\mathbb{N}$ orbits  of $\psi$ on $N.$ Note that the fixed points of $\psi$ correspond to the terms $x_i=1,$ while the lengths of the disjoint cycles in which $\psi$ uniquely splits are given by the terms $x_i\geq 2.$ For instance  $(1\ 2\  3)\in S_3$ has type $[3]$, while $(1\ 2\  3)\in S_4$ has type $[1,3].$ Note also that $T_{id}=[1^n]$. For $\psi \in S_n$,  we denote by $M_{\psi}=\{i\in N :\psi(i)\neq i\}$ the support of $\psi$. Thus $|M_{\psi}|$ is the sum of the terms different form $1$ in $T_{\psi}.$
The permutations of type $[1^{n-k},k]$, for some $k\geq 2$, are  the $k$-cycles; the $2$-cycles are also called transpositions. Note that for every $\psi\in S_n$ and $s\in\mathbb{N},$ $T_{\psi}=T_{\psi^s}$ if and only if $\langle \psi\rangle =\langle \psi^s\rangle. $ Note also that $o(T_{\psi})=o(\psi).$
Recall that $\psi,  \varphi\in S_n$ are conjugate in $S_n$ if and only if $T_{\psi}=T_{\varphi}$. The map $t:S_n\rightarrow \mathfrak{T}(n)$, defined by $t(\psi)=T_{\psi}$  for all $\psi\in S_n$, is surjective. In other words, each partition of $n$ may be viewed as the type of some permutation in $S_n$. We call $t$ the {\it type map}. If $X\subseteq S_n$, then $t(X)$ is the set of types {\it admissible} for $X$ in the sense of \cite[Section 4.1]{bub}, and it is denoted by $\mathfrak{T}(X)$.  For $T\in\mathfrak{T}(n)$, we denote by  $\mu_T(G)$ the number of permutations of type $T$ in $G\leq S_n$. If the normal form of $T$ is $[m_1^{t_1},..., m_k^{t_k}]$, then it is well known that

\begin{equation}\label{count}\mu_T(S_n)=\dfrac{n!}{m_1^{t_1}\cdots m_k^{t_k}t_1!\cdots t_k!}.
\end{equation}

\subsection{\bf {\bf \em{\bf Powers of partitions }}}\label{powers-partition}
\vskip 0.4 true cm

 Given $T=[m_1^{t_1},..., m_k^{t_k}]\in\mathfrak{T}(n)$ in normal form,  the {\it power} of $T$ of exponent $a\in \mathbb{N}$, is defined as the partition
 \begin{equation}\label{power} T^a=\left[\left(\frac{m_1}{\gcd(a,m_1)}\right)^{t_1\gcd(a,m_1)},..., \left(\frac{m_k}{\gcd(a,m_k)}\right)^{t_k\gcd(a,m_k)}\right].
\end{equation}
Note that  $T^a$ is not necessarily in normal form. Moreover, for each $\psi\in S_n$ and each $a\in\mathbb{N},$ we have $T_{\psi^a}=(T_{\psi})^a.$ As a consequence,
the power notation for partitions is consistent with a typical property of the powers: if $a, b\in \mathbb{N},$ then $(T^a)^b=T^{ab}=T^{ba}=(T^b)^a.$ We say that $T^a$ is a {\it proper power} of $T$ if $[1^n]\neq T^a\neq T.$
Throughout the section, we will use the notation in \eqref{power} without further reference.

\begin{lem}\label{trivial-power} Let $T\in \mathfrak{T}(n)$ and $a\in \mathbb{N}$.
\begin{itemize}
\item[(i)] $T^a=T$ if and only if $\gcd(a, o(T))=1.$
\item[(ii)] $T^a=[1^n]$ if and only if $o(T)\mid a.$
\item[(iii)] $T^a$ is a proper power of $T$  if and only if $\gcd(a, o(T))$ is a proper divisor of $o(T).$
\item[(iv)] If $T^a$ is a proper power of $T$, then $o(T^a)$ is a proper divisor of $o(T)$.
\end{itemize}
\end{lem}
\begin{proof}
(i) Let $T^a=T$. Then the number of terms in $T^a$ and in $T$ is the same, that is,  $\sum_{i=1}^kt_i\gcd(a,m_i)=\sum_{i=1}^kt_i$, which implies $\gcd(a,m_i)=1$ for all $i\in\{1,\dots,k\}$ and thus $\gcd(a, \lcm\{m_i\}_{i=1}^k)=\gcd(a, o(T))=1.$ Conversely, if $\gcd(a, o(T))=1,$ then $\gcd(a, \lcm\{m_i\}_{i=1}^k)=1$ and so $\frac{m_i}{\gcd(a,m_i)}=m_i,$ for all $i\in\{1,\dots,k\}$ so that $T^a=T$.

 (ii) $T^a=[1^n]$ is equivalent to $\frac{m_i}{\gcd(a,m_i)}=1$, and thus to $m_i\mid a$  for all $i\in\{1,\dots,k\}$, that is to $o(T)=\lcm\{m_i\}_{i=1}^k\mid a.$

 (iii) It follows by the definition of proper power and by (i) and (ii).

 (iv) Assume that $T^a$ is a proper power of $T$. Then, by (iii), $\gcd(a, o(T))\neq 1,o(T).$
 Let $\psi\in S_n$ such that $T=T_{\psi}$. Then $T^a=T_{\psi^a}$ and $o(T^a)=o(\psi^a)\notin \{ 1,o(\psi)=o(T)\}$ is a proper divisor of  $o(T)$.
\end{proof}

\begin{lem}\label{prospli}
\noindent\begin{itemize}
\item[(i)]  Let $T\in\mathfrak{T}(n)$, $a\in\mathbb{N}$ and $T'=T^a$ be a proper power of $T.$ Then the following facts hold:
\begin{itemize}
\item[(a)] $T'$ admits at least one term with multiplicity at least $2$;
\item[(b)] if there exists a term $x$ of $T'$ appearing with multiplicity $1,$  then  $a$ is coprime with $x$ and $x$ is a term of $T.$
\end{itemize}
\item[(ii)]  Let $h\in\mathbb{N}$, with $1\leq h<n/2.$ Then there exists no type in $\mathfrak{T}(n)$ having $[h,n-k]$ or $[n]$ as a proper power.
\end{itemize}
\end{lem}
\begin{proof} (i)(a) Since $T'=T^a$ is a proper power of $T,$ by  Lemma \ref{trivial-power},
there exists $j\in \{1,\dots,k\}$ with $\gcd(a,m_j)\geq 2.$ Thus the term $\frac{m_j}{\gcd(a,m_j)}$ in $T'$  has multiplicity at least $t_j\gcd(a,m_j)\geq 2.$
(b) If a term $x$ in $T'$ appears with multiplicity $1$, then there exists $j\in \{1,\dots,k\}$ with $t_j\gcd(a,m_j)=1$ and $x=\frac{m_j}{\gcd(a,m_j)}$.
It follows that  $t_j=1$ and $\gcd(a,m_j)=1$. Moreover $x=m_j$ is a term of $T$.

(ii) This is a consequence of (i) and of $k\neq n-k.$
\end{proof}

\subsection{\bf {\bf \em{\bf The type graph of a permutation group }}}\label{power type graph}
\vskip 0.4 true cm
\begin{defn}\label{type-graph}{\rm
Let $G\leq S_n$. We define the {\it  type graph} of $G$, as the graph $\mathcal{T}(G)$ having as vertex set $\mathfrak{T}(G)$ and  $\{T, T'\}\in E_{\mathcal{T}(G)}$ if $T, T'\in \mathfrak{T}(G)$  are one the power of the other.
 We define also the {\it  proper type graph} of $G$, as the $[1^n]$-deleted subgraph of  $\mathcal{T}(G)$ and denote it by $\mathcal{T}_0(G).$ Its vertex set is then $\mathfrak{T}_0(G)=\mathfrak{T}(G)\setminus \{[1^n]\}$ and its edge set is denoted by $E_{\mathcal{T}_0(G)}.$
 The set of its components is denoted by $\mathcal{C}_0(\mathcal{T}(G))$ and their number by $c_0(\mathcal{T}(G))$.}
 \end{defn}
 Note that
 $\{T, T'\}\in E^*_{\mathcal{T}_0(G)}$ if and only if $T, T'\in \mathfrak{T}_0(G)$  are one the  proper power of the other. Clearly $\mathcal{T}(G)$ is always connected and it is $2$-connected if and only if $\mathcal{T}_0(G)$ is connected. Since from now on we intend to focus on proper graphs, we will tacitly assume $n\geq 2.$

 \begin{prop}\label{tau}  Let $G\leq S_n$. Then the following facts hold:
  \begin{itemize}
 \item[(i)] the type map on $G$  induces a complete homomorphism
 $t: \mathcal{\mathcal{P}}_0(G)\rightarrow \mathcal{T}_0(G)$ and  a complete $2$-homomorphism
 $\widetilde{t}: \widetilde{\mathcal{\mathcal{P}}}_0(G)\rightarrow \mathcal{T}_0(G);$
\item[(ii)] $\mathcal{T}_0(G)$ is a quotient of $\widetilde{\mathcal{\mathcal{P}}}_0(G)$. In particular, if $\widetilde{\mathcal{P}}_0(G)$ is connected, then $\mathcal{T}_0(G)$ is connected too.
\end{itemize}
\end{prop}

\begin{proof} (i) We begin showing that the map $t:G\rightarrow \mathfrak{T}(G)$ defines a complete homomorphism   $t:\mathcal{P}(G)\rightarrow \mathcal{T}(G).$ Let $\{\psi, \varphi\}\in E$. Thus $\psi, \varphi\in G$ are one the power of the other; say, $\psi=\varphi^a$ for some $a\in \mathbb{N}.$ Then
 $t(\psi)=T_{\psi}=T_{\varphi^a}=(T_{\varphi})^a=t(\varphi)^a$ and thus $\{t(\psi), t(\varphi)\}\in E_{\mathcal{T}(G)}.$ This shows that $t$ is a homomorphism. We show that $t$ is complete, that is, $t(\mathcal{P}(G))=\mathcal{T}(G)$. Note that $t(G)=\mathfrak{T}(G)$ trivially holds  by definition of $\mathfrak{T}(G)$. Let $e=\{t(\psi), t(\varphi)\}\in  E_{\mathcal{T}(G)}$ so that $t(\psi), t(\varphi)$ are one the power of the other. Let, say, $t(\psi)= t(\varphi)^a$ for some $a\in \mathbb{N}.$ Thus we have $T_{\psi}=T_{\varphi^a}$. Now, define $\overline{\psi}=\varphi^a$ and note that $t(\overline{\psi})=t(\psi)$. Hence $\overline{e}=\{\overline{\psi}, \varphi\}\in E$ and $t(\overline{e})=e.$ Since the only permutation having type $[1^n]$ is $id$, by Lemma \ref{cutgraph-hom}\,(i), we also get a homomorphism $t: \mathcal{P}_0(G)\rightarrow \mathcal{T}_0(G).$ Moreover, applying
Lemma \ref{cutgraph-hom}\,(ii), $t$ also induces the complete homomorphism $\widetilde{t}:\widetilde{\mathcal{P}}(G)\rightarrow \mathcal{T}(G)$ defined by $\widetilde{t}([\psi])=T_\psi$ for all $[\psi]\in [G]$ and a corresponding complete homomorphism $\widetilde{t}: \widetilde{\mathcal{P}}_0(G)\rightarrow \mathcal{T}_0(G).$ It remains to show that $\widetilde{t}$ is a $2$-homomorphism. Let $\{[\psi], [\varphi]\}\in [E]^*_{0}$ and assume that $\widetilde{t}([\psi])=\widetilde{t}([\varphi])$. Then  $[\psi]\neq [\varphi]$ and, by Lemma \ref{lato2}, $\psi$ and  $\varphi$ are one the power of the other.  Moreover, we have $T_{\psi}=T_{\varphi}$. Hence $\langle \psi\rangle=\langle \varphi\rangle$, against $[\psi]\neq [\varphi]$.

(ii)   It follows by (i),  \cite[Lemma 4.3]{bub} and \cite[Proposition 3.2]{bub}.
\end{proof}
Let $G\leq S_n$. The homomorphism $\widetilde{t}$, defined in the above proposition, transfers to $[G]$ all the concepts introduced for $G$ in terms of type. In particular we define
 the type $T_{[\psi]}$ of $[\psi] \in [G]$,  by $T_{\psi}.$ Moreover, we say that  $[\psi]$ is a $k$-cycle (a transposition) if $\psi$ is $k$-cycle (a transposition).
For $X\subseteq S_n$  consider $[X]=\{[x]\in[S_n] :  x\in X\}\subseteq [S_n]$. Then, according to \cite[Section 4.1]{bub}, $\widetilde{t}([X])=t(X)=\mathfrak{T}(X)$
is  the set of  {\it types admissible} for $[X]$.
 Since every subset of  $[S_n]$ is given by $[X]$, for a suitable $X\subseteq S_n,$ that defines the concept of admissibility for all the subsets of $[S_n]$. If $\hat{X}$ is a subgraph of $ \widetilde{\mathcal{P}}_0(G)$ the set of types admissible for $\hat{X}$, denoted  by $ \mathfrak{T}(\hat{X})$,  is given by the set of types admissible for $V_{\hat{X}}.$ In particular, for $C\in \widetilde{\mathcal{C}}_0(G),$ we have $ \mathfrak{T}(C)=\{T\in \mathfrak{T}_0(G):\hbox{there exists}\  [\psi]\in V_C\  \hbox{with}\ T_{\psi}=T\}.$

It is useful to isolate a fact contained in Proposition \ref{tau}.
\begin{cor}\label{edge} Let $G\leq S_n$. If $\varphi,\psi\in G_0$ are such that $\{[\varphi],[\psi]\}\in[E]^*_0$, then one of $T_{\varphi}$ and $T_{\psi}\in\mathfrak{T}_0(G)$ is a proper power of the other. In particular $T_{\varphi}\neq T_{\psi}$.
\end{cor}
Observe that the converse of the above corollary does not hold. Consider, for instance, $\varphi=(1\ 2\ 3\ 4),\  \psi=(1\ 2)(3\ 4)\in S_4.$  We have that $T_{\psi}=[2^2]$ is the power of exponent $2$ of $T_{\varphi}=[4]$, but there is no edge between $[\varphi]$ and $[\psi]$  in $\widetilde{\mathcal{P}}_0(S_4),$ because  no power of $(1\ 2\ 3\ 4)$ is equal to $(1\ 2)(3\ 4).$

\subsection{\bf {\bf \em{\bf The order graph of a permutation group }}}\label{rod-perm}
\vskip 0.4 true cm

In this section we show that,  for any permutation group, the graph $\mathcal{O}_0(G)$ is a quotient of $\mathcal{T}_0(G)$.

Define the map
\begin{equation}\label{omegaT}o_{\mathcal{T}}:\mathfrak{T}_0(G)\rightarrow O_0(G), \qquad o_{\mathcal{T}}(T)=o(T)\  \hbox{for all}\  T\in \mathfrak{T}_0(G)\end{equation}
and recall the map  $\widetilde{o}$ defined in Proposition \ref{order}.

\begin{prop}\label{order-quo} For every $G\leq S_n$, the map $o_{\mathcal{T}}$ defines a complete $2$-homomorphism $o_{\mathcal{T}}:\mathcal{T}_0(G)\rightarrow \mathcal{O}_0(G)$ such that $o_{\mathcal{T}}\circ \widetilde{t}=\widetilde{o}.$
In particular $\mathcal{O}_0(G)$ is a quotient of $\mathcal{T}_0(G)$ and $c_0(\mathcal{O}(G))\leq c_0(\mathcal{T}(G)).$
\end{prop}

\begin{proof} First note that the map $o_{\mathcal{T}}$ in \eqref{omegaT} is well defined. Indeed, if $T\in \mathfrak{T}_0(G)$, then there exists $\psi\in G_0$ such that $T=T_{\psi}$ and so $o(T)=o(T_{\psi})=o(\psi)\in O_0(G).$ The same argument shows that $o_{\mathcal{T}}\circ \widetilde{t}=\widetilde{o}.$  By Proposition \ref{order},  the map $\widetilde{o}$ is a complete homomorphism from $\widetilde{\mathcal{P}}_0(G)$ in $\mathcal{O}_0(G)$. In particular $\widetilde{o}$ is surjective and so also $o_{\mathcal{T}}$ is surjective.

We  show that $o_{\mathcal{T}}$ is a $2$-homomorphism.
Let $\{T,T'\}$ be an edge in $\mathcal{T}_0(G)$. Then  $T,T'$ are one the power of the other. Let, say, $T'=T^a$ for some $a\in \mathbb{N}.$ Then $o(T')$ is a  divisor of $o(T),$ which says $\{o(T),o(T')\}\in E_{\mathcal{O}_0(G)}$.
Moreover, if $T\neq T'$ we have $[1^n]\neq T^a\neq T$ and thus Lemma \ref{trivial-power}\,(iv) implies that $o(T')$ is a proper divisor of $o(T).$

We finally show that $o_{\mathcal{T}}$ is complete. Let $e=\{m,m'\}\in E_{\mathcal{O}_0(G)}$. Since $\widetilde{o}$ is complete, there exist $[\varphi],[\varphi']\in [G]_0$ such that $\widetilde{o}([\varphi])=m, \widetilde{o}([\varphi'])=m'$ and $\{[\varphi],[\varphi']\}\in [E]_0.$ By Proposition \ref{tau}\,(i), the map $\widetilde{t}$ is a homomorphism and so $\{\widetilde{t}([\varphi]),\widetilde{t}([\varphi'])\}=\{T_{\varphi},T_{\varphi'}\}$ is an edge in $\mathcal{T}_0(G).$ Now it is enough to observe that $o_{\mathcal{T}}(T_{\varphi})=o(\varphi)=m$ and $o_{\mathcal{T}}(T_{\varphi'})=o(\varphi')=m'.$
 To close, we apply \cite[ Lemma 4.3 and Proposition 3.2]{bub}.
\end{proof}
\begin{cor}\label{iso-rod-type} Let $G\leq S_n$.  If $m\in O_0(G)$ is isolated in $\mathcal{O}_0(G),$ then each type of order $m$ is isolated in $\mathcal{T}_0(G).$

\end{cor}
\begin{proof}  By Proposition \ref{order-quo}, $o_{\mathcal{T}}$ is a $2$-homomorphism. Thus we can apply to $o_{\mathcal{T}},$  Lemma \ref{2hom-iso}.
\end{proof}

\vskip 0.6 true cm
\section{\bf {\bf \em{\bf Quotient graphs of power graphs associated with fusion controlled permutation groups}}}\label{fusion case}
\vskip 0.4 true cm

\begin{defn}\label{fusion}{\rm Let $G\leq S_n$.
\begin{itemize} \item[(a)] $G$ is called {\it fusion controlled} if $N_{S_n}(G)$ controls the fusion in $G,$ with respect to $S_n$, that is, if for every $ \psi\in G$ and  $x\in S_n$ such that $\psi^x\in G$, there exists $y\in N_{S_n}(G)$ such that $\psi^x=\psi^y.$
\item[(b)]  For each $x\in N_{S_n}(G)$, define the map
\begin{equation}\label{F}
F_{x}: [G]_0\rightarrow[G]_0,\quad F_{x}([\psi])=[\psi^x]\  \hbox{for all} \ [\psi]\in [G]_0.
\end{equation}
\end{itemize}
}
\end{defn}

Note that $F_{x}$  is well defined, that is, for every  $[\psi]\in [G]_0,$ $F_{x}([\psi])$ does not depend on the representative of $[\psi]$,  and $F_{x}([\psi])\in [G]_0$. Those facts are immediate considering that conjugation is an automorphism of the group $S_n$ and that $x\in N_{S_n}(G).$

Recalling  now the homomorphism $\widetilde{t}$ defined in Proposition \ref{tau}  and the definition of   $\mathfrak{G}$-consistency given in \cite[Definition 4.4\,(b)]{bub},
we get the following result.

\begin{prop}\label{conj}  Let $G\leq S_n$. Then the following hold.
\begin{itemize}
\item[(i)] For every $x\in N_{S_n}(G)$, the map $F_x$ is a graph automorphism of $\widetilde{\mathcal{P}}_0(G)$ which preserves the type.
\item[(ii)] If $G$ is fusion controlled, then $\mathfrak{G}=\{F_x :x\in N_{S_n}(G)\}$ is a subgroup of $\mathrm {Aut}(\widetilde{\mathcal{P}}_0(G)) $ and  $\widetilde{t}$ is a $\mathfrak{G}$-consistent $2$-homomorphism from $\widetilde{\mathcal{P}}_0(G)$ to  $\mathcal{T}_0(G)$. In particular $\widetilde{t}$ is a complete orbit $2$-homomorphism.
\end{itemize}
\end{prop}
\begin{proof}(i) Since, for every $x\in  N_{S_n}(G)$, we have $F_x\circ F_{x^{-1}}=F_{x^{-1}}\circ F_x=id_{[G]_0}$ we deduce that $F_x$ is a bijection.
We show that $F_x$ is a graph homomorphism. Let  $\{[\varphi],[\psi]\}\in [E]_0.$ Then, by Lemma \ref{lato2}, $\varphi$ and $\psi$ are one the power of the other. Let , say, $\varphi=\psi^m$ for some $m\in \mathbb{N}$.
Thus, since conjugation is an automorphism of $S_n$,  we have $\varphi^x=(\psi^m)^x=(\psi^x)^m$
and so $\{[\varphi^x],[\psi^x]\}\in [E]_0.$  Next we see that $F_x$ is complete. Let  $e=\{[\varphi^x],[\psi^x]\}\in [E]_0$. Then by Lemma \ref{lato2}, $\varphi^x$ and $\psi^x$ are one the power of the other. Let, say $\varphi^x=(\psi^x)^m$ for some $m\in \mathbb{N}$. Now, considering $\overline{e}=\{[\psi],[\psi^m]\}\in [E]_0$, we have $F_x(\overline{e})=e.$

Finally  $T_{F_x([\psi])}=T_{[\psi]}$  because $\psi^x$ is a conjugate of $\psi$ and thus has its same type. In other words, for every $x\in N_{S_n}(G)$, we have \begin{equation}\label{a}\widetilde{t}\circ F_x=\widetilde{t}.\end{equation}

(ii) The map $N_{S_n}(G)\rightarrow {\mathrm {Aut}}(\widetilde{\mathcal{P}}_0(G))$ associating $F_x$ to $x\in N_{S_n}(G),$ is a group homomorphism and so its image $\mathfrak{G}$ is a subgroup of  ${\mathrm {Aut}}(\widetilde{\mathcal{P}}_0(G))$. We show that $\widetilde{t}$  is $\mathfrak{G}$-consistent checking that conditions (a) and (b) of \cite[Lemma 4.5]{bub} are satisfied. Condition (a) is just  \eqref{a}. To get condition (b), pick $[\varphi],[\psi]\in[G]_0$ with $T_{\varphi}=T_{\psi}$. Then $\varphi$ and $\psi$ are elements of $G$ conjugate in $S_n$. Thus, as $G$ is fusion controlled, they are conjugate also in $N_{S_n}(G)$. So there exists $x\in N_{S_n}(G)$ such that $\varphi=\psi^x$,  which gives $[\varphi]=[\psi^x]=F_x([\psi]).$
\end{proof}

We say that $C, \hat{C}\in \widetilde{\mathcal{C}}_0(G)$ are {\it conjugate} if there exists $x\in N_{S_n}(G)$ such that $ \hat{C}=F_x(C).$

\begin{prop}\label{2con}  Let $G\leq S_n$ be fusion controlled. Then the following hold.
 \begin{itemize}
 \item[(i)]  $\mathcal{T}_0(G)$ is an orbit quotient of $\widetilde{\mathcal{P}}_0(G)$.
 \item[(ii)] $T\in \mathfrak{T}_0(G)$ is isolated in $\mathcal{T}_0(G)$ if and only if each $[\psi]\in [G]_0$ of type $T$ is isolated in $\widetilde{\mathcal{P}}_0(G)$.
  \item[(iii)] If $T\in \mathfrak{T}_0(G)$, then  the components of $\widetilde{\mathcal{P}}_0(G)$ admissible for $T$ are conjugate.
 \end{itemize}
\end{prop}
\begin{proof} Keeping in mind the definition of orbit quotient given in \cite[Definition 5.14]{bub}, (i) is a rephrasing of Proposition \ref{conj}\,(ii). To show (ii) recall that, by \cite[Proposition 5.9\,(ii)]{bub},  every orbit homomorphism is locally surjective and then
apply \cite[Corollary 5.12]{bub} and  Lemma \ref{2hom-iso} to the orbit $2$-homomorphism $\widetilde{t}.$ (iii) is an application of \cite[Proposition 6.9\,(i)] {bub}.
 \end{proof}

  \begin{lem}\label{tc} Let $G$ be fusion controlled, $C\in\widetilde{\mathcal{C}}_0(G)$ and  $T\in \mathfrak{T}(C)$. Then $\mathfrak{T}(C)=V_{C(T)}.$
\end{lem}
\begin{proof} Apply  \cite [Theorem A\,(i)]{bub} to  $\widetilde{t}$ recalling that, by \eqref{general-inc},  every orbit homomorphism is locally surjective.
  \end{proof}

Proposition \ref{conj} guarantees that  when $G$ is fusion controlled, we have
\begin{equation}\label{ttilde-prop}
\widetilde{t}\in \mathrm{O}\big (\widetilde{\mathcal{P}}_0(G),\mathcal{T}_0(G)\big )\cap\mathrm{Com}\big(\widetilde{\mathcal{P}}_0(G),\mathcal{T}_0(G)\big).
\end{equation}
 Thus the machinery for counting the components of $\widetilde{\mathcal{P}}_0(G)$ by those in $\mathcal{T}_0(G)$ can start provided that we control the numbers $k_{\widetilde{\mathcal{P}}_0(G)}(T)$ and $k_{C}(T)$, for $T\in\mathfrak{T}(G_0)$ and $C\in\widetilde{\mathcal{C}}_0(G)$. We see first  that  $k_{\widetilde{\mathcal{P}}_0(G)}(T)$ is easily determined, for any permutation group $G$, through $\mu_{T} (G).$ Recall that $\phi$ denotes the Euler totient function.
 \begin{lem}\label{lem:2}
Let  $G\leq S_n$. If $T\in \mathfrak{T}(G_0),$ then  $k_{\widetilde{\mathcal{P}}_0(G)}(T)=\dfrac{\mu_{T} (G)}{\phi(o(T))}$.
\end{lem}
\begin{proof}
If  $T=[m_1^{t_1},..., m_k^{t_k}]\in \mathfrak{T}(G_0)$, then the set $G_T=\{\sigma\in G: \sigma\  \hbox{is of type }\ T\}$ is nonempty and each element in $G_T$ has the same order given by $\mathrm{lcm}\{m_i\}_{i=1}^k=o(T).$ Consider the equivalence relation $\sim$  which defines the quotient graph  $\widetilde{\mathcal{P}}_0(G)$.
 Since generators of the same cyclic subgroup of $G$ share the same type, it follows that $G_T$ is union of $\sim$-classes, each of them of size $\phi(o(T))$. On the other hand, $[\sigma]\in [G]$ has type $T$ if and only if $\sigma\in G_T.$ This means that $k_{\widetilde{\mathcal{P}}_0(G)}(T)$ is the number of $\sim$-classes contained in $G_T,$ and so $k_{\widetilde{\mathcal{P}}_0(G)}(T)=\dfrac{\mu_{T} (G)}{\phi(o(T))}$.
 \end{proof}

 Let $T\in\mathfrak{T}_0(G)$. Accordingly to \cite[Definition 6.1]{bub}, we consider the set $\widetilde{\mathcal{C}}_0(G)_T$ of the components of $\widetilde{\mathcal{P}}_0(G)$ admissible for $T$ and  denote by $\widetilde {c}_0(G)_{T}$ its order.
Thus $\widetilde {c}_0(G)_{T}$ counts the number of components of $\widetilde{\mathcal{P}}_0(G)$ in which there exists at least one vertex $[\psi]$ of type $T.$

\begin{lem}\label{fusion2} Let $G\leq S_n$ be fusion controlled and $T\in\mathfrak{T}_0(G)$.
\begin{itemize}
\item[(i)] Then $\widetilde {c}_0(G)_{T}=\frac{k_{\widetilde{\mathcal{P}}_0(G)}(T)}{k_C(T)}=\frac{\mu_{T} (G)}{\phi(o(T))k_C(T)}$ for all $C\in \widetilde{\mathcal{C}}_0(G)_T.$
 \item[(ii)] If $T$ is isolated in $\mathcal{T}_0(G)$, then $\widetilde {c}_0(G)_{T}=k_{\widetilde{\mathcal{P}}_0(G)}(T)=\frac{\mu_{T} (G)}{\phi(o(T))}.$
  \end{itemize}
  \end{lem}
  \begin{proof}(i) By  \eqref{ttilde-prop} and \eqref{general-inc}, we can apply \cite[Proposition 6.8]{bub} to $\widetilde{t}$ and then use
  Lemma \ref{lem:2} to make the computation explicit.

  (ii) is a consequence of (i) and of  Proposition \ref{2con}\,(ii).
\end{proof}

\begin{proof} [Proof of Theorem A]  Applying \cite[Theorem A]{bub}  to $\widetilde{t}$, we get
 \[c_0(G)=\widetilde {c}_0(G)=\sum_{i=1}^{c_0(\mathcal{T}(G))}\widetilde {c}_0(G)_{T_i}\]

and, by Lemma \ref{fusion2}\,(i),  we obtain
\[\sum_{i=1}^{c_0(\mathcal{T}(G))}\widetilde {c}_0(G)_{T_i}=\sum_{i=1}^{c_0(\mathcal{T}(G))}\frac{\mu_{T_i} (G)}{\phi(o(T_i))k_{C_i}(T_i)}.\]
\end{proof}

 We now observe some interesting limitations to the types  which can appear in a same component of $\widetilde{\mathcal{P}}_0(G)$.

\begin{cor}\label{two-types}  Let $G\leq S_n$ be fusion controlled and $C\in \widetilde{\mathcal{C}}_0(G)$.
\begin{itemize}
\item[(i)] If $T,T'\in  \mathfrak{T}(C),$ then $\frac{k_{\widetilde{\mathcal{P}}_0(G)}(T)}{k_C(T)}=\frac{k_{\widetilde{\mathcal{P}}_0(G)}(T')}{k_C(T')}.$
\item[(ii)] $C\cong \widetilde{t}(C)$ if and only if
there exists $T\in \mathfrak{T}(C)$ such that $k_{C}(T)=1$ and, for every $T'\in \mathfrak{T}(C)$, $k_{\widetilde{\mathcal{P}}_0(G)}(T)=k_{\widetilde{\mathcal{P}}_0(G)}(T')$.
\item[(iii)] If $C$ contains all the vertices of $[G]_0$ of a certain type $T\in \mathfrak{T}(C)$, then $C$ also contains all the vertices of $[G]_0$ of  type $T'$ for all $T'\in \mathfrak{T}(C).$
\item[(iv)] If there exists $T\in \mathfrak{T}(C)$ such that $k_{C}(T)=k_{\widetilde{\mathcal{P}}_0(G)}(T)>1,$ then $C\not\cong \widetilde{t}(C).$
\end{itemize}
\end{cor}
\begin{proof}  An immediate application of \cite[Proposition 7.2]{bub}.
\end{proof}

Finally we state a useful criterion of connectedness for $\widetilde{\mathcal{P}}_0(G)$.

\begin{cor}\label{2conreverse}  Let $G$ be fusion controlled and let $\mathcal{T}_0(G)$ be connected. Then $\widetilde{\mathcal{P}}_0(G)$ is a union of conjugate components. Moreover, if there exist $T\in \mathfrak{T}_0(G)$ and $C\in \widetilde{\mathcal{C}}_0(G)$, with $C$ containing all the vertices of $[G]_0$ of type $T,$ then $\widetilde{\mathcal{P}}_0(G)$ is connected.
\end{cor}
\begin{proof}
Let $C\in  \widetilde{\mathcal{C}}_0(G)$ and consider $\widetilde{t}(C).$ Then $\widetilde{t}(V_C)=\mathfrak{T}(C).$
By \cite[Theorem A\,(i)] {bub}, $\widetilde{t}(C)$ is a component of $\mathcal{T}_0(G)$ and since that graph is connected, we get $\widetilde{t}(V_C)=\mathfrak{T}_0(G)$ and so $\widetilde{t}^{-1}(\widetilde{t}(V_C))=[G]_0.$
By \cite[Theorem A\,(iii)] {bub}, this implies that $\widetilde{\mathcal{P}}_0(G)$ is the union of the components in $\mathcal{C}(\widetilde{\mathcal{P}}_0(G))_{\widetilde{t}(C)}$ which, by \cite[Lemma 6.3\,(i)] {bub} are equal to the components admissible for any $T \in \mathfrak{T}(G_0)$. So Proposition \ref{2con}\,(iii) applies giving $\widetilde{\mathcal{P}}_0(G)$ as  a union of conjugate components. The last part follows from \cite[Corollary 5.15]{bub}.
 \end{proof}

\vskip 0.6 true cm
\section{\bf {\bf \em{\bf The number of components in  $\widetilde{\mathcal{P}}_0(S_n),\  \mathcal{T}_0(S_n),\ \mathcal{ O}_0(S_n)$}}}
\vskip 0.4 true cm

In this section we apply Procedure \cite[6.10]{bub} to $S_n$ (which is fusion controlled) to make the computation of  Formula \eqref{formula-fusion22} concrete.
Exploiting the strong link among the graphs $\mathcal{P}_0(S_n),\ \widetilde{\mathcal{P}}_0(S_n),\  \mathcal{T}_0(S_n),\ \mathcal{ O}_0(S_n)$, we determine simultaneously $c_0(S_n)=\widetilde {c}_0(S_n)$, $c_0(\mathcal{T}(S_n))$ and $c_0(\mathcal{ O}(S_n))$. On the route, we give a description of the components of those four graphs.
Recall that $\mathfrak{T}_0(n)= \mathfrak{T}(S_n)\setminus\{[1^n]\}$ and that, for $T\in \mathfrak{T}_0(n)$,  the numbers $\mu_T(S_n)$ are computed by \eqref{count}. Also recall that, if $T\in \mathfrak{T}_0(n)$ and $C$ is a component of $\widetilde{\mathcal{P}}_0(S_n)$, then $k_{C}(T)$ counts the number of vertices in $C$ having type $T.$

By Lemma \ref{tc} and Theorem A, the procedure  to get  $\widetilde {c}_0(S_n)$ translates into the following.

 \begin{proc} {\bf Procedure to compute $\widetilde {c}_0(S_n)$ }\label{procedureS_n}
{\rm
\begin{itemize}\item[ (I)] {\it  Selection of $T_i$ and $C_i$}
\end{itemize}
\begin{itemize}
\item[ {\it Start}]  : Pick arbitrary $T_1\in \mathfrak{T}_0(n)$ and choose any $C_1\in\widetilde{\mathcal{C}}_0(S_n)_{T_1}$.

\item[ {\it Basic step}]: Given  $T_1,\dots, T_i\in \mathfrak{T}_0(n) $ and $C_1,\dots,C_i\in\widetilde{\mathcal{C}}_0(S_n)$ such that $C_j\in \widetilde{\mathcal{C}}_0(S_n)_{T_j}$ ($1\leq j\leq i$), choose  any $T_{i+1}\in \mathfrak{T}_0(n) \setminus \bigcup_{j=1}^i \mathfrak{T}(C_{j})$ and any  $C_{i+1}\in \widetilde{\mathcal{C}}_0(S_n)_{T_{i+1}}.$

\item[ {\it Stop}]: The procedure stops in $c_0(\mathcal{T}(S_n))$ steps.
\end{itemize}
\medskip

\begin{itemize}\item[ (II)] {\it  The value of  $\widetilde {c}_0(S_n)$ }
\end{itemize}
 Compute the integers $$\widetilde {c}_0(S_n)_{T_j}=\frac{\mu_{T_j} (S_n)}{\phi(o(T_j))k_{C_j}(T_j)}\quad  (1\leq j\leq c_0(\mathcal{T}(S_n)))$$ and sum them up to get  $\widetilde {c}_0(S_n)$.}
 \end{proc}
 The complete freedom in the choice of the
$C_j\in \widetilde{\mathcal{C}}_0(S_n)_{T_j}$  allows us
to compute  each $ \widetilde{c}_0(S_n)_{T_j}={\frac{\mu_{T_j} (S_n)}{\phi(o(T_j)) k_{C_j}(T_j)}}$
selecting $C_j$ as the component containing $[\psi]$, for $[\psi]$ chosen
as preferred among the $\psi\in S_n\setminus\{id\}$ with $T_{\psi}=T_j.$ We will
apply  this fact with no further mention.  We emphasize also that the computation  is made easy by \eqref{count}.
Remarkably, the number
$c_0(\mathcal{T}(S_n))$ counts the steps of the procedure.

\vskip 0.4 true cm
\subsection{\bf  Preliminary lemmas and small degrees }
\vskip 0.4 true cm

We start by summarizing what we know about isolated vertices by Proposition \ref{2con}\,(ii), Lemma \ref{isolated-order} and Corollary \ref{iso-rod-type}.

\begin{lem} \label{isolatedSn}\noindent \begin{itemize}\item[(i)] The type $T\in \mathfrak{T}_0(S_n)$ is isolated in $\mathcal{T}_0(S_n)$ if and only if each $[\psi]\in [S_n]_0$ of type $T$ is isolated in $\widetilde{\mathcal{P}}_0(S_n)$.
 \item[(ii)] If $m\in O_0(S_n)$ is isolated in $\mathcal{O}_0(S_n),$ then each vertex of order $m$ is isolated in $\widetilde{\mathcal{P}}_0(S_n)$ and each type of order $m$ is isolated in $\mathcal{T}_0(G).$
 \item[(iii)] If, for some $\psi\in S_n$, $[\psi]$ is isolated in $\widetilde{\mathcal{P}}_0(S_n)$, then $o(\psi)$ is prime and the component of $P_0(S_n)$ containing $\psi$ is a complete graph on $p-1$ vertices.
 \end{itemize}
\end{lem}

As a consequence, we are able to analyze the prime or prime plus $1$ degrees.

\begin{lem}\label{lem:3}
Let  $n\in\{p, p+1\}$ for some $p\in P$. Then the following facts hold:
 \begin{itemize}
 \item[(i)] $p$ is isolated in $\mathcal{O}_0(S_n).$ The type $[1^{n-p},p]$ is isolated in $\mathcal{T}_0(S_n);$
  \item[(ii)] each vertex  of $[S_n]_0$ of order $p$ is  isolated in $\widetilde{\mathcal{P}}_0(S_n)$;
  \item[(iii)] the number of components of $\widetilde{\mathcal{P}}_0(S_n)$  containing the elements of order $p$ in $[S_n]_0$ is given by $\widetilde {c}_0(S_p)_{[p]}=(p-2)!$ if $n=p$, and by $\widetilde {c}_0(S_{p+1})_{[1,p]}=(p+1)(p-2)!$ if $n=p+1$.
 \end{itemize}
 \end{lem}
\begin{proof} (i)-(ii)  Since  $n\in\{p, p+1\}$, we have $p\leq n$ so that $S_n$ admits elements of order $p$. Since there exists no element with order $kp$ for $k\geq 2$, $p$ is isolated in $\mathcal{O}_0(S_n).$ Thus Lemma \ref{isolatedSn} applies.

 (iii) The counting follows from Lemma \ref{fusion2}\,(ii) and Formula \eqref{count} after having observed that the only type of order $p$ in $\mathcal{T}_0(S_p)$ is $[p]$ and that the only type of order $p$ in $\mathcal{T}_0(S_{p+1})$ is $[1,p]$.
\end{proof}
\begin{lem}\label{lem:4}
For $n\geq 6$, the transpositions of $\widetilde{\mathcal{P}}_0(S_n)$ lie  in the same component $\widetilde{\Delta}_n$ of $\widetilde{\mathcal{P}}_0(S_n)$.
 Moreover $\mathfrak{T}(\widetilde{\Delta}_n)\supseteq \{[1^{n-2},2][1^{n-5},2,3], [1^{n-3},3]\}.$
\end{lem}
\begin{proof}
Let $[\varphi_1]$ and $[\varphi_2]$ be two distinct transpositions in $S_n$. Then their supports $M_{\varphi_1}$ $M_{\varphi_2}$ are distinct. If
$|M_{\varphi_1}\cap M_{\varphi_2}|=1$, then there exist distinct $a,b,c\in N$ such that $\varphi_1=(a~b)$ and $\varphi_2=(a~c)$. Moreover, as $n\geq 6$, there exist distinct $e,f,g\in N\setminus \{a,b,c\}$ and we have the path
$$[(a~b)], [(a~b)(d~e~f)], [(d~e~f)], [(a~c)(d~e~f)], [(a~c)]$$
between $[\varphi_1]$ and $[\varphi_2]$.
If $|M_{\varphi_1}\cap M_{\varphi_2}|=0$, then there exist distinct $a,b,c, d\in N$ such that $\varphi_1=(a~b)$ and $\varphi_2=(c~d).$ Let $\varphi_3=(a~c)$. By the previous case, there exists a path between $[\varphi_1]$ and
$[\varphi_3]$ and a path between $[\varphi_2]$ and $[\varphi_3]$. Therefore there exists also a path between $[\varphi_1]$ and $[\varphi_2]$. This shows that all the transpositions of $\widetilde{\mathcal{P}}_0(S_n)$ lie  in the same component $\widetilde{\Delta}_n.$ Next, collecting the types met in the paths, we get $\mathfrak{T}(\widetilde{\Delta}_n)\supseteq \{[1^{n-2},2][1^{n-5},2,3], [1^{n-3},3]\}.$
\end{proof}
We note now an interesting immediate fact.
\begin{lem}\label{complete}  Let $X, Y$ be graphs and $f\in\mathrm{Hom}(X,Y)$. If $\hat{X}$ is a complete subgraph of $X$, then $f(\hat{X})$ is a complete subgraph of $Y.$
\end{lem}

\begin{cor}\label{comp-structure} Let $n\geq 6$ and $\Delta_n$ be the unique component of $\mathcal{P}_0(A_n)$ such that $\pi (\Delta_n)=\widetilde{\Delta}_n.$ Then neither one of the components $\Delta_n,\  \widetilde{\Delta}_n,\  \widetilde{t}(\widetilde{\Delta}_n)$ of the graphs $\mathcal{P}_0(S_n)$, $\widetilde{\mathcal{P}}_0(S_n)$, $\mathcal{T}_0(S_n)$
respectively, nor  the connected subgraph $\widetilde{o}(\widetilde{\Delta}_n)$ of $\mathcal{O}_0(S_n)$ is a complete graph.
\end{cor}
\begin{proof}  First note that  the existence of a unique component  $\Delta_n$ of  $\mathcal{P}_0(A_n)$ such that $\pi(\Delta_n)=\widetilde{\Delta}_n$ is guaranteed by \cite[Corollary 5.13]{bub}, because $\pi$ is pseudo-covering and tame due to Lemma \ref{rmk:1}. Moreover, by Proposition \ref{conj}, $\widetilde{t}$ is a complete orbit homomorphism and thus locally surjective. Hence \cite[Theorem A\,(i)] {bub} guarantees that  $\widetilde{t}(\widetilde{\Delta}_n)$ is a component of $\widetilde{\mathcal{P}}_0(S_n)$ with  $V_{\widetilde{t}(\widetilde{\Delta}_n)}=\mathfrak{T}(\widetilde{\Delta}_n)$. On the other hand, by Proposition \ref{order-quo}, 
we have the complete graph homomorphism $o_{\mathcal{T}}:\mathcal{T}_0(S_n)\rightarrow \mathcal{O}_0(S_n)$ such that $o_{\mathcal{T}}\circ \widetilde{t}=\widetilde{o}$. In particular $\widetilde{o}(\widetilde{\Delta}_n)=o_{\mathcal{T}}(\widetilde{t}(\widetilde{\Delta}_n)),$ so that
 we can interpret the sequence  of graphs
 $$\Delta_n,\  \widetilde{\Delta}_n,\  \widetilde{t}(\widetilde{\Delta}_n),\ \widetilde{o}(\widetilde{\Delta}_n)$$
  as
   \begin{equation}\label{new}\Delta_n,\  \pi(\Delta_n),\ ( \widetilde{t}\circ \pi)(\Delta_n),\ (o_{\mathcal{T}}\circ \widetilde{t}\circ \pi)(\Delta_n).
   \end{equation}
By Lemma \ref{lem:4}, $\widetilde{o}(\widetilde{\Delta}_n)$ admits as vertices the integer $2$ and $3$ which are not adjacent in $\mathcal{O}_0(S_n)$. Thus
$\widetilde{o}(\widetilde{\Delta}_n)$ is not a complete graph. Then to deduce that no graph in the sequence \eqref{new} is complete we start from the bottom and apply three times Lemma \ref{complete}. 
\end{proof}
Note that, in general, $\widetilde{o}(\widetilde{\Delta}_n)$ is not a component of $\mathcal{O}_0(S_n)$ because $\widetilde{o}$ is not pseudo-covering. For instance, $\widetilde{o}(\widetilde{\Delta}_6)$ is not a component of $\mathcal{O}_0(S_6)$ because $4\notin V_{\widetilde{o}(\widetilde{\Delta}_6)}$ while $4$ belongs to the component of $\mathcal{O}_0(S_6)$ containing $\widetilde{o}(\widetilde{\Delta}_6)$.
An argument in the proof of Theorem B\,(ii) shows that $\widetilde{o}(\widetilde{\Delta}_n)$ is indeed a component at least for $n\geq8$.

\begin{proof} [Proof of Theorem B\,(i)] For each $n$ ($2\leq n\leq 7$), we compute $\widetilde{c}_0(S_n), c_0(\mathcal{T}(S_n))$ and $c_0(\mathcal{O}(S_n))$ separately. We view $S_n$ as acting on $N=\{1,\dots,n\}$.

Since $[S_2]_0=\{ [(1~2)]\}$, we immediately have $\widetilde{c}_0(S_2)=c_0(\mathcal{T}(S_2))=c_0(\mathcal{O}(S_n))=1.$  Since $[S_3]_0=\{ [(1~2)], [(1~3)], [(2~3)],[(1~2~3)]\}$, we have $\mathfrak{T}_0(S_3)=\{[1,2], [3]\}$ and  $O_0(S_3)=\{2, 3\}$. Thus, by Lemma \ref{lem:3}, we get $\widetilde{c}_0(S_3)=4$ and $c_0(\mathcal{T}(S_3))=c_0(\mathcal{O}(S_n))=2.$

Let $n=4$. We start considering the type $T_1=[4]$ and the cycle $\psi=(1~2~3~4) \in S_4$. By Corollary \ref{edge} and Lemma \ref{prospli}, the only vertex distinct from $[\psi]$ adjacent to $[\psi]$ is $\varphi=[(1~3)(2 ~4)]$ and no other vertex can be adjacent to $[\psi]$ or $[\varphi ]$. Thus the component  $C_1$ of $\widetilde{\mathcal{P}}_0(S_4)$ having as  a vertex  $[\psi]$ is a path of length one, $k_{C_1}(T_1)=1$ and  $\widetilde{c}_0(S_4)_{T_1}=\frac{\mu_{[4]}(S_4)}{\phi(4)}=3$. Note that $\mathfrak{T}(C_1)=\{[4], [2^2]\}.$
 By Lemma \ref{lem:3}, a vertex of type $T_2=[1, 3]$ is isolated and thus $\widetilde {c}_0(S_4)_{T_2}=\frac{\mu_{[1,3]}(S_4)}{\phi(3)}=4.$
Consider now  the type $T_3=[1^2, 2]$. $T_3$ is not a proper power and has no proper power. So, by Corollary \ref{edge}, a component admissible for $T_3$ is again reduced to a single vertex. Thus $\widetilde {c}_0(S_4)_{T_3}=\mu_{[1^2, 2]}(S_4)=6$. Since all the possible types in $S_4$ have been considered, the Procedure  \ref{procedureS_n} ends, giving $c_0(\mathcal{T}(S_4))=3$ and $\widetilde{c}_0(S_4)=3+4+6=13$. Since $O_0(S_4)=\{2,4,3\}$ we instead have $c_0(\mathcal{O}(S_4))=2.$

 Let $n=5$. By Lemma~\ref{lem:3}\,(ii), the vertices of type $T_1=[5]$ in $\widetilde{\mathcal{P}}_0(S_5)$ are in $3!=6$ components which are isolated vertices. Let $C_1$ be one of those components.
Consider for the type $T_2=[1,4]$,  the cycle $\psi=(1~2~3~4) \in S_5$ . By Corollary \ref{edge} and Lemma \ref{prospli}, the component $C_2$ containing $[\psi]$ admits as vertices just $[(1~2~3~4)]$ and $[(1~3)(2~4)].$ Thus $k_{C_2}(T_2)=1$ and  $\widetilde {c}_0(S_5)_{T_2}=\frac{\mu_{[1,4]}(S_5)}{\phi(4)}=15$. Moreover, $\mathfrak{T}(C_1)\cup\mathfrak{T}(C_2)=\{[5], [1,4], [1,2^2]\}.$ We next consider $T_3=[2,3]$ and $\psi=(1~2)(3~4~5)\in S_5$. By Lemma \ref{prospli}, $T_3$ is not a power and so there exists no $[\varphi]\in [S_5]_0$ such that $\varphi^s=\psi$. On the other hand, to get a power of $T_3=T_{\psi}$ different from $[1^5]$, we must consider $T_{\psi^a}$ where $\gcd(a, o(\psi))\neq 1$, that is, $\psi^a$ for $a\in \{2,3,4\}.$ Thus the component $C_3$ containing $[\psi]$ contains the path $[(1~2)],[(1~2)(3~4~5)],[(3~4~5)].$ We show that  $C_3$ is indeed that path.
 If there exists a proper edge $\{[(3~4~5)], [\varphi]\}$, then, by Corollary \ref{edge}, $T_{\varphi}\notin \mathfrak{T}(C_1)\cup\mathfrak{T}(C_2)\cup\{[1^2,3]\}$ and thus $T_{\varphi}\in \{[1^3, 2], [2,3]\}.$ But $[1^3, 2]$ and $[1^2,3]$ are not one the power of the other and thus, by Corollary \ref{edge}, we must have $T_{\varphi}=[2,3]$, say $\varphi=(a~b)(c~d~e)$ where $\{a, b, c, d, e\} =N$. So $\left[(a~b)(c~d~e)\right]^2=(c~e~d)$ is a generator of $\langle (3~4~5)\rangle$. It follows that $(a~b)=(1~2)$ and $(c~e~d)\in \{(3~4~5), (3~5~4)\}$. Thus  $\varphi\in \{\varphi_1=(1~2)(3~4~5), \varphi_2=(1~2)(3~5~4)\}.$ But it is immediately checked that
 $[\varphi_1]=[\varphi_2]=[\psi].$
 Similarly one can check that the only proper edge $\{[(1~2)], [\varphi]\}$ is given by the choice $\varphi=\psi.$
 So we have $\widetilde {c}_0(S_5)_{T_3}=\frac{\mu_{[2,3]}(S_5)}{\phi(6)}=10$ and,  since all the possible types in $S_4$ have been considered, we get $c_0(\mathcal{T}(S_5))=3$ and $\widetilde{c}_0(S_5)=31$. On the other hand there are only two components for $\mathcal{O}_0(S_5)$: one reduced to the vertex $5$ and the other one having as set of vertices $\{2,3,4,6\}.$

Let $n=6$. In $\widetilde{\mathcal{P}}_0(S_6)$, by Lemma~\ref{lem:3}\,(iii), the elements of type $T_1=[1, 5]$ are inside $36$ components which are isolated vertices. Let $C_1$ be one of them. Consider the type $T_2=[2,4]$ and $\psi=(1~2)(3~4~5~6)$. The component $C_2$ containing $[\psi]$ is the path
$$[(1~2)(3~4~5~6)],\  [(3~5)(4~6)],\ [(3~4~5~6)].$$

This is easily checked, by Corollary \ref{edge}, taking into account that the only proper power of $[2,4]$ is $[1^2,2^2]$ and that, by Lemma \ref{prospli}, no type admits $[2,4]$ as proper power. Moreover, $[1^2,2^2]$  admits no proper power  and is the proper power only of $[2,4]$ and $[1^2,4].$
It follows that $k_{C_2}(T_2)=1$ and  $\widetilde {c}_0(S_6)_{T_2}=\frac{\mu_{[2,4]}(S_6)}{\phi(4)}=45$.  Since $$\mathfrak{T}(C_1)\cup\mathfrak{T}(C_2)=\{[1,5], [2,4], [1^2,2^2], [1^2,4]\},$$ we  consider $T_3=[1^4,2]$ and $\psi=(1~2)\in S_6$.  By Lemma~\ref{lem:4}, all the vertices of type $T_3$ are in $C_3=\widetilde{\Delta}_6$ and, using Corollary \ref{two-types}\,(iii),  we see that $C_3$ contains also all the vertices of type $[1,2,3]$ and $[1^3,3]$.
But it is easily checked that no further type exists having as power one of the types $[1^4,2], [1,2,3], [1^3,3]$, so that $\mathfrak{T}(C_3)=\{[1^4,2], [1,2,3], [1^3,3]\}$. We claim that all the elements of type $T_4=[2^3]$ are in a same component $C_4,$ so that $\widetilde {c}_0(S_6)_{T_4}=1.$
Let  $[\varphi_1]$ and $[\varphi_2]$ be distinct elements in $[S_6]_0$ of type $[2^3]$. Note that, since $[\varphi_1], [\varphi_2]$ are distinct they share at most one transposition. Let $\varphi_1=(a~b)(c~d)(e~f)$, with $\{a, b, c, d, e, f\}=\{1, 2, 3, 4, 5, 6\}$. Since the $2$-cycles in which $\varphi_1$ splits commute and also the entries in each cycle commute, we can restrict our analysis to $\varphi_2=(a~b)(c~e)(d~f)$, if $\varphi_1, \varphi_2$ have one cycle in common, and to $\varphi_2=(a~c)(b~e)(d~f)$, if $\varphi_1, \varphi_2$ have no cycle in common. In the first case we have the following path of length $8$ between $[\varphi_1]$ and $[\varphi_2]$:
$$[\varphi_1], [(a~e~d~b~f~c)], [(a~d~f)(e~b~c)], [(e~a~b~d~c~f)], [(e~d)(a~c)(b~f)],$$
$$ [(d~a~b~e~c~f)], [(d~b~c)(a~e~f)], [(a~d~e~b~f~c)], [\varphi_2].$$
In the second case we have the following path of length $4$ between $[\varphi_1]$ and $[\varphi_2]$:
$$[\varphi_1], [(a~f~d~b~e~c)], [(a~d~e)(f~b~c)], [(f~a~b~d~c~e)], [\varphi_2].$$
Collecting the types met in those paths,  we see that $\mathfrak{T}(C_4)\supseteq\{[2^3], [6], [3^2]\}$ and since all the other possible types in $S_6$ have been considered we get that $\mathfrak{T}(C_4)=\{[2^3], [6], [3^2]\}$. Thus
our procedure ends giving  $c_0(\mathcal{T}(S_6))=4$ and $\widetilde{c}_0(S_6)=83$. Moreover $c_0(\mathcal{O}(S_6))=2$ with the two components of $\mathcal{O}_0(S_6)$ having as vertex sets $\{5\}$ and $\{2,3,4,6\}.$

Let $n=7.$ In $\widetilde{\mathcal{P}}_0(S_7)$, by Lemma~\ref{lem:3}, the elements of type $T_1=[7]$ are in $\widetilde {c}_0(S_7)_{T_1}=120$ components which are isolated vertices. Let $C_1$ be one of them.
By Lemma~\ref{lem:4} and Corollary \ref{two-types}\,(iii),  all the vertices of type $T_2=[1^5,2]$  and those of types $[1^2,2,3], [1^4,3]$
are in the same component $C_2=\widetilde{\Delta}_7.$  In particular $\widetilde {c}_0(S_7)_{T_2}=1.$ We show that also the types $[1^3, 4]$, $[1^3, 2^2]$, $[2^2, 3]$, $[1, 2, 4]$, $[2, 5]$ and $[1^2, 5]$ are admissible for  $C_2$.  From the path
$$ [(1~2~3~4)],\  [(1~3)(2~4)],\  [(1~3)(2~4)(5~6~7)],\  [(5~6~7)]$$
we deduce that $[1^3, 4]$, $[1^3, 2^2]$ and $[2^2, 3]$ are admissible for $C_2$, because $[(5~6~7)]$ is a vertex of  $C_2$. Then it is enough to consider the path
$ [(1~2)(3~4~5~6)], [(3~5)(4~6)]$ for getting the type $[1, 2, 4]$
and the path
$[(1~2~3~4~5)], [(1~2~3~4~5)(6~7)], [(6~7)]$ for getting the types $[2, 5]$ and $[1^2, 5]$.

We turn now our attention to the type $T_3=[1, 6]$ and to $\psi_0=(1~2~3~4~5~6)\in S_7.$ Let  $C_3$ be the component of $\widetilde{\mathcal{P}}_0(S_7)$, containing $[\psi_0].$ By Lemma \ref{prospli}\,(ii), $T_3$ is not a power and its only power are the types $[1, 2^3], [1, 3^2]$, which in turn admit no powers and are only the power of $T_3$.
It follows that $\mathfrak{T}(C_3)=\{[1, 2^3], [1, 3^2],[1, 6] \}.$  In particular, if $[\psi]\in V_{C_3}$, then  $\psi$ admits a unique fixed point. Moreover, by what shown for $[S_6]$, all the vertices in $[S_7]$ of type $T_3$ fixing $7$ are  contained in  $C_3$.  We show that, indeed, each $[\psi]\in V_{C_3}$ is such that $\psi(7)=7.$
By contradiction, assume $\psi(7)\neq 7$, for some $[\psi]\in V_{C_3}$. Then, since there is a path between $[\psi]$ and $[\psi_0]$, there exist $[\varphi], [\psi']\in V_{C_3}$ with $\varphi(7)=7$, $\psi'(j)=j $,  for some $j\neq 7$ and an edge $\{[\varphi], [\psi']\}\in [E]^*_{0}$. But either $\varphi$ is a power of $\psi'$ and so  $\varphi$ admits $j$ as a fixed point or  $\psi'$  is a power of $\varphi$ and so $\psi'$ admits $7$ as  a fixed point. In any case, we reach a contradiction. Thus we have $k_{C_3}(T_3)=\frac{\mu_{[6]}(S_6)}{\phi(6)}$ and so $\widetilde {c}_0(S_7)_{T_3}=\frac{\mu_{[1,6]}(S_7)}{\mu_{[6]}(S_6)}=7$. Since there are no other types left in $S_7$, we conclude that $c_0(\mathcal{T}(S_7))=3$ and $\widetilde{c}_0(S_7)=128$. Moreover $c_0(\mathcal{O}(S_7))=2$ with the two components of $\mathcal{O}_0(S_7)$ having as vertex sets $\{7\}$ and $\{2,3,4,5,6,10\}.$
\end{proof}

We are ready to show that, for  $n \geq 8,$ the main role is played by the component $\widetilde{\Delta}_n$ defined in Lemma \ref{lem:4}.

\begin{prop}\label{lem:8} For $n \geq 8$, all the vertices of $\widetilde{\mathcal{P}}_0(S_n)$ apart from those of prime order $p\geq n-1$ are contained in $\widetilde{\Delta}_n.$
\end{prop}
\begin{proof} Let $n \geq 8$. We start showing that each
$[\psi] \in [S_n]_0$ having even order is a vertex of $\widetilde{\Delta}_n.$ Let
$o(\psi)=2k$, for  $k$ a positive integer.  Since $o(\psi^k)=2$,  $\psi^k$ is the product of $s\geq 1$ transpositions. If $s=1,$ then we  have $\psi^k=(a~b)$ for suitable $a,b\in N$ and, by Lemma \ref{lem:4},  the path $[\psi], [(a~b)]$  has its end vertex in $\widetilde{\Delta}_n$.  If $s=2,$ then
$\psi^k=(a~b)(c~d)$,  for suitable $a, b, c, d \in
N$. Since $n\geq 8$, there exist distinct
 $e, f, g \in N\setminus \{a, b, c, d\}$ and we have the path
 $$[\psi],[\psi^k], [(a~b)(c~d)(e~f~g)], [(e~f~g)], [(a~b)(e~f~g)], [(a~b)]$$
 with an end vertex belonging to $\widetilde{\Delta}_n$.
Finally if $s\geq 3$, we have $\psi^k=(a~b)(c~d)(e~f)\sigma$, for suitable
$a, b, c, d, e, f \in N$ and $\sigma\in S_n$, with $\sigma^2=id$. Let $\varphi=(a~c~e~b~d~f)\sigma$, so that $\varphi^3=\psi^k$. Since $n\geq 8$, then there exist distinct $g, h \in N\setminus\{a, b, c, d, e, f\}$ and we have the  path
$$[\psi], [\psi^k], [\varphi], [(a~e~d)(c~b~f)], [(a~e~d)(c~b~f)(g~h)], [(g~h)]$$
with an end vertex belonging to  $\widetilde{\Delta}_n$.

Next let  $o(\psi)=p$, where $p$ is an odd prime such that  $p\leq n-2$.  If $|M_{\psi}|\leq n-2,$ pick $a, b\in \{1, 2, \dots, n\}\setminus M_{\psi}$ and consider the path $[\psi], [\psi(a~b)], [(a~b)].$ If $|M_{\psi}|\geq n-1$, observe that, since $p\leq n-2$,
$\psi$ is the product of $s\geq 2$ cycles of length $p$, say $\psi=(a_1~a_2~\dots~a_p)(b_1~b_2~\dots~b_p)\sigma$, where $\sigma=id$ or $\sigma$ is the product of $s-2$ cycles of length $p$. Let $\varphi=(a_1~b_1~a_2~b_2~...~a_p~b_p){\sigma}^{\frac{(p+1)}{2}}$. Since $o(\varphi)$ is even, by what shown above,  we get $[\varphi] \in V_{\widetilde{\Delta}_n}$. Moreover we have  $\varphi^{2}=\psi$ and thus $[\psi]\in V_{\widetilde{\Delta}_n}$.

Finally let $o(\psi)=upq$, where $p, q\geq 3$ are distinct prime numbers and $u$ is
an odd positive integer.  Then $o(\psi^u)=pq$ and in the split of $\psi^u$ into disjoint cycles, there exists either a cycle of length $pq$ or two cycles of length $p$ and $q$.  In the first case $pq\leq n$ gives $p\leq n/q\leq n/3< n-2$. In the second case we have $p+q\leq n$, which gives $p\leq n-q< n-2.$
Thus  $o(\psi^{uq})=p< n-2$ and, by the previous case, we obtain $[\psi^{uq}]\in V_{\widetilde{\Delta}_n}$ so that also $[\psi]\in V_{\widetilde{\Delta}_n}$.
\end{proof}

\begin{proof}  [Proof of Theorem B\,(ii) and Theorem D]   Let $n\geq 8$ be fixed and recall that
 $c_0(S_n)=\widetilde {c}_0(S_n)$. First of all, observe that $P\cap O_0(S_n)=\{p\in P: p\leq n\}.$ Therefore we can reformulate Lemma \ref{lem:8} by saying that all the vertices in $[S_n]_0$ of order not belonging to the set $B(n)=P\cap \{n, n-1\}$ are in $\widetilde{\Delta}_n.$   In particular, $V_{\widetilde {o}(\widetilde{\Delta}_n)}\supseteq O_0(S_n)\setminus B(n).$
Let $\Sigma_n$ be the  unique component  of $\mathcal{O}_0(S_n)$
such that $V_{\Sigma_n}\supseteq V_{\widetilde {o}(\widetilde{\Delta}_n)}$.

If $n\notin P\cup (P+1)$, then $B(n)=\varnothing$ and thus all the vertices of $\widetilde{\mathcal{P}}_0(S_n)$  are in $\widetilde{\Delta}_n$. In this case $\widetilde{\mathcal{P}}_0(S_n)=\widetilde{\Delta}_n$ is connected and, by \cite[Proposition 3.2]{bub}, both its quotients $\mathcal{T}_0(S_n)$  and $\mathcal{O}_0(S_n)$ are connected.
Note that the completeness of  $\widetilde{t}$ and $\widetilde {o}$ implies  $\widetilde{t}( \widetilde{\Delta}_n )=\mathcal{T}_0(S_n)$ and $\widetilde {o}(\widetilde{\Delta}_n)=\mathcal{O}_0(S_n)=\Sigma_n.$

Next let $n\in P\cup (P+1)$, say $n=p$ or $n=p+1$ for some $p\in P,$ necessarily odd. Then $B(n)=\{p\}$.
We need to understand only the components of $\widetilde {\mathcal{P}}_0(S_n)$ and $\mathcal{T}_0(S_n)$
containing vertices of order $p$, and decide whether or not $p\in V_{\Sigma_n}$.
By Lemma~\ref{lem:3}, $p$ is isolated in $\mathcal{O}_0(S_p)$ and the unique type  $T$ of order $p$ is isolated in $\mathcal{T}_0(S_n).$  Moreover
each component of $\widetilde{\mathcal{P}}_0(S_n)$ admissible for $T$ is an isolated vertex and their number is known. The values for $c_0(S_n)$ as displayed in Table \ref{eqtable2} and the fact that $c_0(\mathcal{T}(S_n))=2$ immediately follows. Now note that  $\Sigma_n$ cannot reduce to the isolated vertex $p$, because $\Sigma_n$ contains at least the vertex $2\in O_0(S_n)\setminus \{p\}$. 
Thus $c_0(\mathcal{O}_0(S_n))=2$  and the two components of $\mathcal{O}_0(S_n)$ have as vertex sets $\{p\}$ and $V_{\Sigma_n}= O_0(S_n)\setminus \{p\}=V_{\widetilde {o}(\widetilde{\Delta}_n)}.$ Since we have shown  that $\widetilde{\Delta}_n$ is the only possible component of $\widetilde{\mathcal{P}}_0(S_n)$ not reduced to an isolated vertex and  $\widetilde {o}$ is complete, applying \cite[Proposition 5.2]{bub}  we get $\Sigma_n=\widetilde {o}(\widetilde{\Delta}_n).$

So far, for every $n\geq 8$, we have shown  that $\widetilde{\Delta}_n$ is the only possible component of $\widetilde{\mathcal{P}}_0(S_n)$ not reduced to an isolated vertex; $\widetilde{t}(\widetilde{\Delta}_n)$  is the only possible component of $\mathcal{T}_0(S_n)$ not reduced to an isolated vertex; $\widetilde {o}(\widetilde{\Delta}_n)$ is the only possible component of $\mathcal{O}_0(S_n)$ not reduced to an isolated vertex.
For $\mathcal{P}_0(S_n)$,  $\widetilde{\mathcal{P}}_0(S_n)$, $\mathcal{T}_0(S_n)$ and $\mathcal{O}_0(S_n)$, the main component shall respectively refer to the component $\Delta_n$ such that $\pi(\Delta_n)=\widetilde{\Delta}_n$ defined in Corollary \ref{comp-structure}, $\widetilde{\Delta}_n$,  $\widetilde{t}(\widetilde{\Delta}_n)$ and
$\widetilde {o}(\widetilde{\Delta}_n)$.
Then, by Corollary \ref{comp-structure}, no main component is complete.
Finally we show that every component $C$ of $\mathcal{P}_0(S_n)$, with $C\neq \Delta_n$ is a complete graph on $p-1$ vertices. Let $\psi\in V_C$. Then $[\psi]\notin V_{\widetilde{\Delta}_n}$ and thus $[\psi]$ is isolated in $\widetilde{\mathcal{P}}_0(S_n)$. Hence, to conclude, we  invoke Lemma \ref{isolatedSn}\,(iii).

\end{proof}

\begin{cor}\label{isoclass} The following are equivalent:
\begin{itemize}
\item[(i)] every component $C$ of $\widetilde{\mathcal{P}}_0(S_n)$ is isomorphic to the component of $\mathcal{T}_0(S_n)$ induced on $\mathfrak{T}(C)$;
\item[(ii)] $2\leq n\leq 5.$
\end{itemize}
\end{cor}
\begin{proof} (ii)$\Rightarrow $(i) For $2\leq n\leq 5$, the case-by-case proof of Theorem B\,(i) shows directly the required isomorphism.

(i)$\Rightarrow $(ii) For $n\geq 6,$ we show that the component $\widetilde{\Delta}_n$ is not isomorphic to the component of $\mathcal{T}_0(S_n)$ induced on $\mathfrak{T}(\widetilde{\Delta}_n)$. Namely $\widetilde{\Delta}_n$ contains all the vertices of type $T=[1^{n-2},2]$ and since $k_{\widetilde{\mathcal{P}}_0(S_n)}(T)=\frac{n(n-1)}{2}>1,$ Corollary \ref{two-types}\,(iv) applies.
\end{proof}

\begin{proof}  [Proof of Corollary C] A check on Tables \ref{eqtable1} and \ref{eqtable2} of Theorem B.
 \end{proof}
 \begin{cor} \label{final} Apart from the trivial case $n=2$, the minimum $n\in\mathbb{N}$ such that $\mathcal{P}(S_n)$ is $2$-connected is $n=9$. There exists infinitely many $n\in\mathbb{N}$ such that $\mathcal{P}(S_n)$ is $2$-connected.
 \end{cor}
 \begin{proof}  Let $n=k^2$, for  some $k\geq 3.$ Then $n\geq 8, n\notin P$ and $n-1=k^2-1=(k-1)(k+1)$ is not a prime. Thus, by Theorem B, $\mathcal{P}(S_n)$ is $2$-connected. In particular $c_0(S_9)=1$. Moreover by Tables \ref{eqtable1} and \ref{eqtable2}, we have $c_0(S_n)>1$ for all $3\leq n\leq 8.$
 \end{proof}

{\bf Acknowledgements}  The authors wish to thank Gena Hahn and Silvio Dolfi for some suggestions on a preliminary version of the paper.
The first author is partially supported by GNSAGA of INdAM.

\bigskip
\bigskip
{\footnotesize \pn{\bf Daniela~Bubboloni}\; \\  {Dipartimento di Matematica e Informatica U.Dini},\\
{ Viale Morgagni 67/a}, {50134 Firenze, Italy}\\
{\tt Email:  daniela.bubboloni@unifi.it}\\

{\footnotesize \pn{\bf Mohammad~A.~Iranmanesh}\; \\ {Department of
Mathematics},\\ {Yazd University,  89195-741,} { Yazd, Iran}\\
{\tt Email: iranmanesh@yazd.ac.ir}\\

{\footnotesize \pn{\bf Seyed~M.~Shaker}\; \\  {Department of
Mathematics}, \\ {Yazd University,  89195-741,} { Yazd, Iran}\\
{\tt Email: seyed$_{-}$shaker@yahoo.com}\\


\begin{thebibliography}{20}
\bibitem{sur} J. Abawajy, A. Kelarev and M. Chowdhury, Power Graphs: A Survey, {\em Electronic Journal of Graph Theory and Applications }{\bf1} (2) (2013), 125-147.
\bibitem{bub} D.~Bubboloni, Graph homomorphisms and components of quotient graphs, {\em Rend. Sem. Mat. Univ. Padova} (to appear). Available at arXiv: 1605.03549v2 (2016).
\bibitem{BIS2} D. Bubboloni, Mohammad A. Iranmanesh, S. M. Shaker, On some graphs associated with the finite alternating groups, arXiv: 1412.7324v2 (2016).
\bibitem{c} P.~J.~Cameron, The power graph of a finite group II, {\em J. Group Theory} {\bf 13} (2010), 779-783.
\bibitem{cg} P.~J.~Cameron and S.~Ghosh, The power graph of a finite Group, {\em Discrete Mathematics} {\bf 311} (2011), 1220-1222.
\bibitem{cgs} I.~Chakrabarty, S.~Ghosh and M.~K.~Sen, Undirected power graphs of semigroups, {\em Semigroup Form} {\bf 78} (2009), 410-426.
\bibitem{pya2} B. Curtin, G. R. Pourgholi and H. Yousefi-Azari, On the punctured power graphs of finite groups,
{\em Australiasian J. Combinatorics} {\bf 62} (2015), 1-7.
\bibitem{dl} R. Diestel, {\em  Graph Theory}, Graduate Texts in Mathematics Vol. 173, Springer-Verlag Heidelberg, 2010.
\bibitem{DF} A. Doostabadi and M. Farrokhi D. G., On the connectivity of proper power graphs of finite groups, {\em Comm. Algebra} {\bf 43} (2015), 4305-4319.
\bibitem{kq} A.~V.~Kelarev and S.~J.~Quinn, A combinatorial property and power graphs of groups,
{\em Contrib. General Algebra} {\bf 12} (2000), 229-235.
\bibitem{kna} U. Knauer,  {\it Algebraic Graph Theory: Morphisms, Monoids and Matrices}, Walter de Gruyter, 2011.
\bibitem{kna-paper} U. Knauer, Endomorphism spectra of graphs, {\em Discrete Mathematics} {\bf 109} (1992), 45-57.
\bibitem{man} M.~Mirzargar, A.~R.~Ashrafi and M.~J.~Nadjafi-Arani, On the power graph of a finite group, {\em Filomat} {\bf 26} (6) (2012), 1201-1208.
\bibitem{MRS} A. R. Moghaddamfar, S. Rahbariyan and W. J. Shi, Certain properties of the power graph associated with a finite group, {\em J. Algebra  Appl.} {\bf 13} (2014), DOI: 10.1142/S0219498814500406.
\bibitem{pya} G.~R.~Pourgholi, H.~Yousefi-Azari and A.~R.~Ashrafi, The Undirected Power Graph of a Finite Group, {\em  Bull. Malays. Math. Sci. Soc.} {\bf 38} (2015), 1517-1525.
 \bibitem{IS} S. M. Shaker, Mohammad A. Iranmanesh, On groups with specified quotient power graphs, {\em Int. J. Group Theory}, Vol. {\bf 5} No. 3 (2016), 49-60.

\end{thebibliography}
\end{document}